\documentclass{amsart}

\usepackage{booktabs}
\usepackage{amsmath}
\usepackage[all]{xy}
\usepackage{graphicx}
\usepackage{verbatim}
\usepackage[danish,english]{babel}
\usepackage{color}
\usepackage{enumerate}
\usepackage[abs]{overpic}
\usepackage[makeroom]{cancel}
\usepackage{amssymb}
\usepackage{bold-extra}
\usepackage{rotating}

\newcommand{\C}{\mathbb C}

\newcommand{\R}{\mathbb R}
\newcommand{\N}{\mathbb N}
\newcommand{\Z}{\mathbb Z}

\newcommand{\imply}{\Longrightarrow}

\newcommand{\fD}{\mathcal D}

\newcommand{\fO}{\mathcal O}

\newcommand{\im}{\mathrm{Im }}

\newcommand{\supp}{\mathrm{supp }}

\newcommand{\Tr}{\mathrm{Tr}}

\newtheorem{thm}{Theorem}[section]

\newtheorem{lem}[thm]{Lemma}

\theoremstyle{definition}

\newtheorem{defn}[thm]{Definition}  
\newtheorem{exam}[thm]{Example}     

\reversemarginpar

\renewcommand{\qed}{\hfill$\square$}

\newcommand{\Addresses}{{
  \bigskip
  \footnotesize
  S\o ren Haagerup, \textsc{IMADA, University of Southern Denmark, Campusvej 55, 5230 Odense M, Denmark.}\par\nopagebreak
  \textit{E-mail address}: \texttt{shaagerup@imada.sdu.dk}
  \medskip
  \medskip

  Uffe Haagerup. 

  \medskip
  \medskip

  Maria Ramirez-Solano, \textsc{IMADA, University of Southern Denmark, Campusvej 55, 5230 Odense M, Denmark.}\par\nopagebreak
  \textit{E-mail address}: \texttt{solano@imada.sdu.dk}
}}

\usepackage[hidelinks]{hyperref}
\begin{document}

\title[Computational explorations of $T$ for the amenability problem of $F$]
 {Computational explorations of the Thompson group $T$ for the amenability problem of $F$}

\author[ S.~Haagerup~U.~Haagerup~M.~Ramirez-Solano ]{S. Haagerup, U. Haagerup$^*$, M. Ramirez-Solano$^\dag$.}
\date{June 23,2018}



\thanks{$^*$Supported by the Villum Foundation under the project ``Local and global structures of groups and their algebras" at University of Southern Denmark, and by the ERC Advanced Grant no. OAFPG 247321, and partially supported by the Danish National Research Foundation (DNRF) through the Centre for Symmetry and Deformation at University of Copenhagen, and the Danish Council for Independent Research, Natural Sciences.}
\thanks{$^\dag$ Supported by the Villum Foundation under the project ``Local and global structures of groups and their algebras" at University of Southern Denmark, and by the ERC Advanced Grant no. OAFPG 247321, and by the Center for Experimental Mathematics at University of Copenhagen.}

\begin{abstract}
It is a long standing open problem whether the Thompson group $F$ is an amenable group.
In this paper we show that if $A$, $B$, $C$ denote the standard generators of Thompson group $T$ and  $D:=C B A^{-1}$ then
$$\sqrt2+\sqrt3\,<\,\frac1{\sqrt{12}}||(I+C+C^2)(I+D+D^2+D^3)||\,\le\, 2+\sqrt2.$$
Moreover, the upper bound is attained if the Thompson group $F$ is amenable.
Here, the norm of an element in the group ring $\C T$ is computed in $B(\ell^2(T))$ via the regular representation of $T$. 
Using the ``cyclic reduced" numbers $\tau(((C+C^2)(D+D^2+D^3))^n)$, $n\in\N$, and some methods from our previous paper \cite{HaagerupRamirezSolano} we can obtain precise lower bounds as well as good estimates of the spectral distributions of $\frac1{12}((I+C+C^2)(I+D+D^2+D^3))^*(I+C+C^2)(I+D+D^2+D^3),$ where   $\tau$ is the tracial state on the group von Neumann algebra $L(T)$.
Our extensive numerical computations suggest that 
$$\frac{1}{\sqrt{12}}||(I+C+C^2)(I+D+D^2+D^3)||\approx 3.28,$$
and thus that $F$ might be non-amenable.
However, we can in no way rule out that $\frac1{\sqrt{12}}||(I+C+C^2)(I+D+D^2+D^3)||=\, 2+\sqrt2$.
\end{abstract}

\maketitle
\section{\textbf{Introduction}}\label{s:one}
\begin{defn}
  The Thompson group $T$ is the group of (cyclic) order preserving homeomorphisms $f:\R/\Z\to \R/\Z$ for which:
  \begin{itemize}
    \item $f$ and $f^{-1}$ are piecewise linear with finitely many breakpoints.
    \item all breakpoints of $f$ and $f^{-1}$ are in $\Z[\tfrac 12]/\Z$.
    \item all slopes of $f$ are in the set $2^{\Z}:=\{2^n\mid n\in\Z\}$.
  \end{itemize}
\end{defn}
$T$ is a countable group. It is generated by the elements $C,D$, whose graphs are shown in Fig. \ref{f:generatorsCDofT}.
\begin{figure}
\includegraphics[scale=.50]{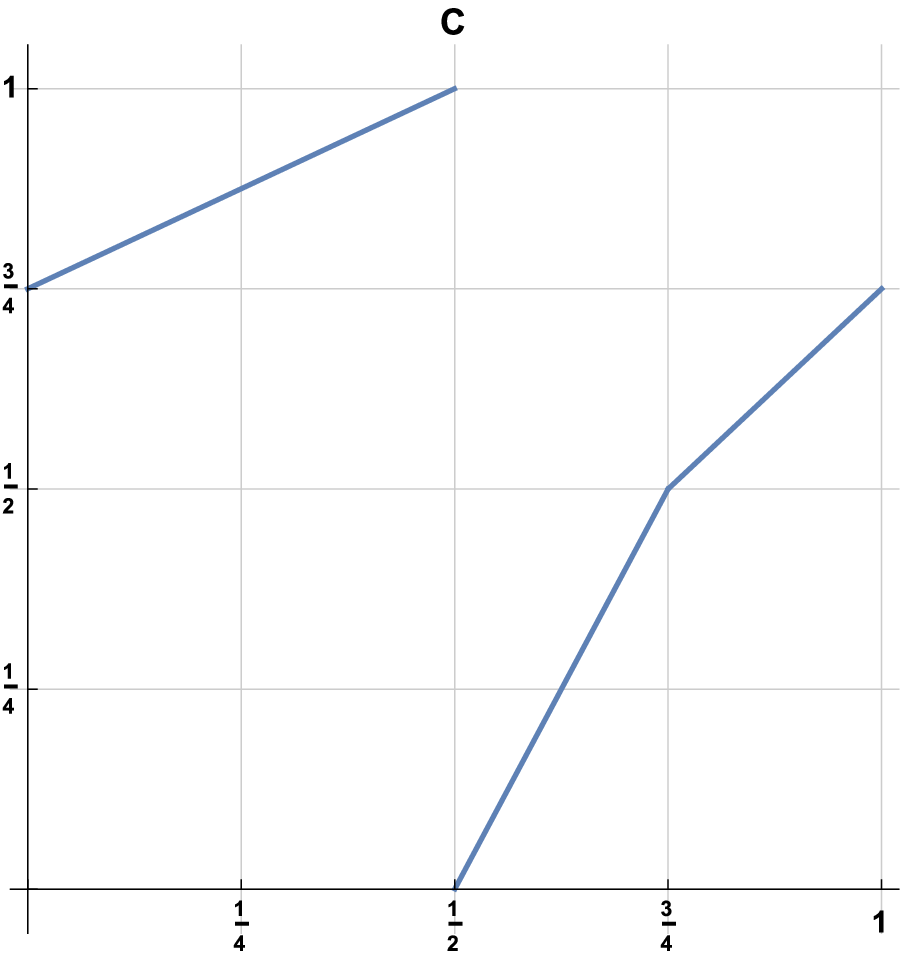}\qquad
\includegraphics[scale=.50]{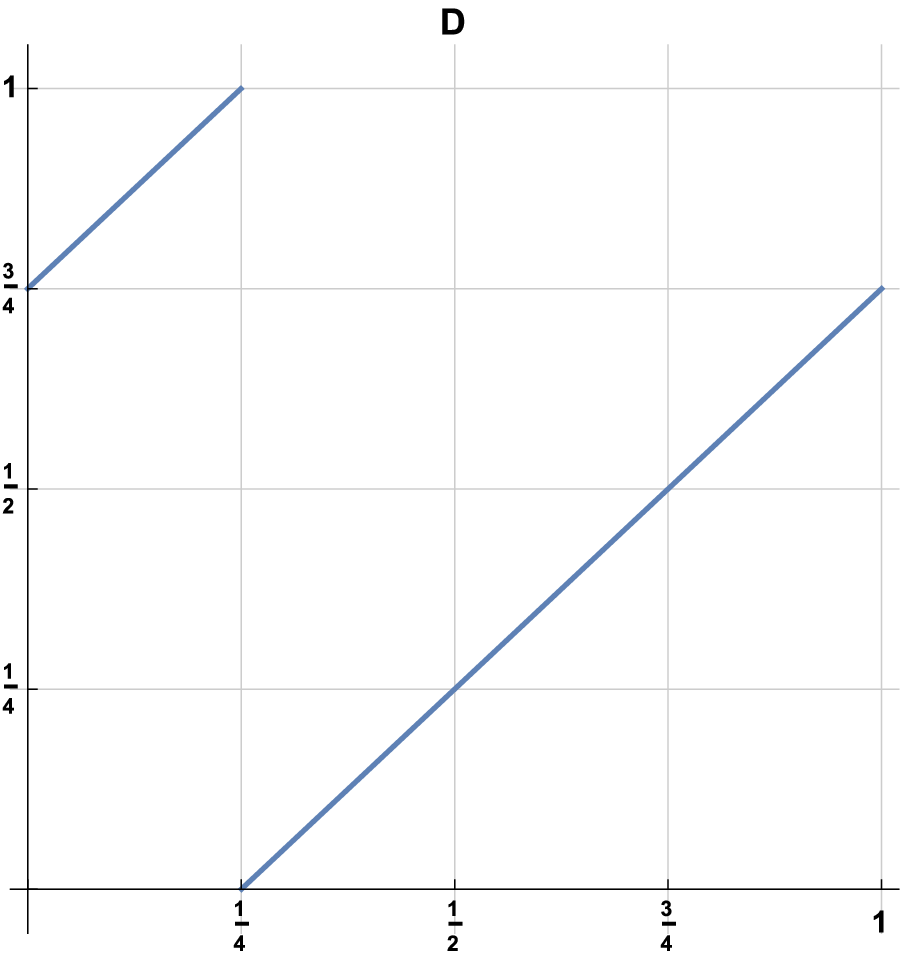}\\
  \caption{The generators $C,D$ of the Thompson group $T$.}\label{f:generatorsCDofT}
\end{figure}
Moreover, it has the finite presentation
\begin{multline}
T=\langle C,D\mid C^3=D^4=(CD)^5=e,\\  [CDC, D^2(CDC)D^2]=[D^2(CDC)D^2,C(D^2CDCD^2)C^{-1}]=e\rangle,
\end{multline}
where the group commutator is given by the definition $[g,h]=ghg^{-1}h^{-1}$.
If $A,B,C$ denote the standard generators of $T$ \cite{CFP}, then $D=C B A^{-1}$.

Recall that $F$ is the subgroup of $T$ given by
$$F=\langle A,B\rangle =\{ f\in T \mid f(0)=0\}.$$  
It is a long standing open problem whether the Thompson group $F$ is amenable. 
The present paper is a continuation of the work started in \cite{HaagerupRamirezSolano} by the same authors.
In that paper, we tested the amenability problem of $F$ by
estimating norms of certain elements in the group ring $\C F$ using computers. 
Thanks to the new algorithms we devised to compute words in $C,D$ in polynomial time (see Section \ref{s:amenability}) and to some results by Haagerup and Olesen \cite{HO}, we can now test the amenability problem of $F$ by computing norms of certain elements in $\C T$.
Extrapolations of our computational results suggest the same as our previous paper, namely, that $F$ might not be amenable. A recent experimental work by Elder, Rechnitzer and Janse~van~Reusburg \cite{ElderRechnitzerJanse} on the amenability problem of $F$ using statistical methods arrives also to the same conclusion that $F$ might be non-amenable (see also \cite{ElderRogers}). 

As in \cite[Section 2]{HaagerupRamirezSolano},  by the norm of $||a||$  of an element $a$ in the group ring of a discrete group $\Gamma$ we mean
$$||a||=||\lambda(a)||_{B(\ell^2(\Gamma))},$$
where $\lambda$ is the left regular representation of $\Gamma$. We continue using the standard convention of writing  $a\in B(\ell^2(\Gamma))$ instead of $\lambda(a)\in B(\ell^2(\Gamma))$ for any $a\in \C \Gamma$.
Our starting point is the following theorem (see Section \ref{s:amenability} for more details) 
\begin{thm}Let $C,D$ be the generators of T, whose graphs are shown in Fig. \ref{f:generatorsCDofT}, and let I denote the unit element of $T$.  Then
$$\sqrt2+\sqrt3\,<\,\frac1{\sqrt{12}}||(I+C+C^2)(I+D+D^2+D^3)||\,\le\, 2+\sqrt2.$$
Moreover, the upper bound is attained if the Thompson group $F$ is amenable.
\end{thm}
In comparison, one gets for the free product $\Z_3\star\Z_4$ on two generators $c\in\Z_3, d\in\Z_4$ that
\begin{equation}\label{e:hfree}
\frac1{\sqrt{12}}||(e+c+c^2)(e+d+d^2+d^3)||\,=\,\sqrt2+\sqrt3.
\end{equation} 
(cf.~Section \ref{s:free}).
Given a discrete group $\Gamma$, let $L(\Gamma)$ denote its  von Neumann algebra. That is, $L(\Gamma)$ is the von Neumann algebra in $B(\ell^2(\Gamma))$ generated by $\Gamma$. Then we can define the normal faithful tracial state $\tau$ on $L(\Gamma)$ by
$$\tau(h):=\langle h\delta_e,\delta_e\rangle, \quad h\in L(\Gamma).$$
 (see e.g. \cite[Section 6.7]{KadRin}). 
 Moreover, we can express the norm of $h$ in terms of the moments
$$m_n(h^*h):=\tau((h^*h)^n),\qquad n\in\N_0,$$
 namely, (cf.~\cite[Section 4]{HaagerupRamirezSolano})
 $$||h||=||h^*h||^{1/2}=\lim_{n\to\infty}\tau((h^*h)^n)^{\frac1{2n}}.$$
 The challenge is to compute all the moments $m_n(h^*h)$, and in practice we can only compute a finite number of them.
In this paper, we were able to compute the moments $m_n(h^*h)$, $n=0,\ldots,28$ for the element 
$$h=\frac1{\sqrt{12}}(I+C+C^2)(I+D+D^2+D^3),$$
using efficient methods, both mathematically and computationally.
This procedure can be adapted to elements in a discrete group $\Gamma$ that are expressed similarly.
Define the self-adjoint operator
$$\tilde h:=\left(
\begin{array}{cc}
 0 & h^* \\
 h & 0 \\
\end{array}
\right)\in M_2(B(\ell^2(T))),
$$
and let $\tilde \tau :=\tau\otimes \tau_2$, where $\tau_2:=\frac12Tr$ on $M_2(\C)$.
Then 
$$\Lambda:C(\sigma(\tilde h))\to \C$$
given by $\Lambda(f):=\tilde \tau(f(\tilde h))$ is a positive functional on $C(\sigma(\tilde h))$ such that $\Lambda(1)=1$.
Hence there is a unique Borel probability measure for which
$$\tilde\tau(f(\tilde h)) =\int_{\sigma(\tilde h)} f\,d\mu,\quad \forall f\in C(\sigma(\tilde h))$$ 

Since $\tilde\tau$ is faithful, the support supp$(\mu)=\sigma(\tilde h)\subset[\,-||h||,\,\,||h||\,]$.
Such measure $\mu$ is invariant under the reflection $t\mapsto -t$ because the odd moments of $\tilde h$ are zero, and it satisfies
$$\int_{\,-||h||}^{||h||} t^{2n} d\mu(t)=m_{2n}(\tilde h)=m_n(h^*h),\qquad n\in\N_0$$
(cf.~\cite[Section 2]{HaagerupRamirezSolano}).
Moreover, $\pm||h||\in\mathrm{supp}(\mu)$ by symmetry of $\mu$. 
The theory of orthonormal polynomials applied to this measure  (cf.~\cite[Section 4]{HaagerupRamirezSolano}) together with the moments yield an increasing sequence of numbers converging to $||h||$.
We can calculate the first 28 of these numbers using our computed moments, and a suitable extrapolation of these numbers ($n=0,\ldots,28$) gives
$$3.2016\le\frac1{\sqrt{12}}||(I+C+C^2)(I+D+D^2+D^3)||\approx 3.28.$$ 
We found that this sequence of numbers actually gives much better lower bounds than the sequence of ``roots" and ``ratios" of the moments that also converge to the norm (cf.~Section \ref{s:results}). 
Furthermore, we also estimate the Lebesgue density of $\mu$ with fairly high precision, which shows that the measure $\mu$ is very close to zero on the interval $[3.22,2+\sqrt2]$. However, we cannot rule out that the measure has very ``thin tail" stretching all the way up to $2+\sqrt2$.

 In comparison, one gets  Eq.~(\ref{e:hfree}) when one considers the free product $\Z_3\star \Z_4$ on two generators $c\in\Z_3,d\in\Z_4$.
The measure $\mu$ based on $c,d$ instead of $C,D$ will be denoted by $\mu_{\mathrm{free}}$, and it can be computed explicitly (See Section \ref{s:free})
\begin{equation}\label{e:mufree}
\mu_{\mathrm{free}}=\frac{1}{2\pi}\frac{\sqrt{24-(x^2-5)^2}}{x(12-x^2)}1_{[-\sqrt2-\sqrt3,\sqrt2-\sqrt3]\cup[\sqrt3-\sqrt2,\sqrt2+\sqrt3]}(x)\,dx+\frac34\delta_0.
\end{equation}

In our previous paper, we estimated the norm $3.60613\le||A+B+A^{-1}+B^{-1}||\approx 3.87$, where $A,B$ are the standard generators of $F$, by using the first 24 moments of $(A+B+A^{-1}+B^{-1})^2$. Elvey-Price \cite{EP} succeeded in computing 7 more moments, which pass the test of \cite[Theorem B.1]{HaagerupRamirezSolano} for possible computational errors.
 Using these 31 moments, the updated estimated norm remains unchanged, while the new lower bound is
$$3.64271\le||A+B+A^{-1}+B^{-1}||\approx 3.87,$$
with a very likely lower bound of $3.70211$, based on the same list of moments.

In comparison, 
$$3.7873\le||C+D+C^{-1}+D^{-1}||\approx 3.84,$$
which suggests that the actual norm is also much closer to $4$, even though it cannot be $4$ because the Thompson group $T$ is not amenable \cite{BrinS}. This norm is estimated in the same way as we did with the norm $||A+B+A^{-1}+B^{-1}||$ (cf.~Appendix \ref{a:normCDCinvDinv}). The composition of words in $C$, $D$ is given in Section \ref{s:amenability}, and it was inspired by the Belk and Brown forest algorithm in \cite{BelkBrownForestDiagrams} (see also \cite{Burillo}).

\section{\textbf{Action of $C$,$D$, the cyclic reduced numbers $\zeta_n$, and amenability.}}\label{s:amenability}
Since the elements of the Thompson group $T$ are piecewise linear functions with breakpoints on the diadics, it makes sense to introduce the set of diadic intervals
$$\fD:=\{[\frac{k}{2^n} , \frac{k + 1}{2^n} ] \mid n \in \N_0 , k = 0,\ldots , 2^n-1\}.$$
Note that the unit interval is the largest diadic interval, and cutting a diadic interval into two equal halves gives two new diadic intervals. Then we write the elements of $T$ as functions $\fD\to\fD$.
For instance, the element $C$ whose graph is in Figure \ref{f:generatorsCDofT} is determined by the function 
$$f_{C}:\{d_1, d_2, d_3\}\to \{r_1,r_2,r_3\}$$
given by
$$f_C(d_1)=r_3\qquad f_C(d_2)=r_1\qquad f_C(d_3)=r_2,$$
where the diadic intervals
$d_1:=[0,\frac12]$, $d_2:=[\frac12,\frac34]$, $d_3:=[\frac34,1]$ in this case are obtained from the breakpoints of $C$ and note that they only overlap on their boundaries. Similarly $r_1:=[0,\frac12]$, $r_2:=[\frac12,\frac34]$, $r_3:=[\frac34,1]$ are obtained from the breakpoints of $C^{-1}$.
 
Observe that $\fD$ is not closed under union, e.g. the union of the diadic intervals $[\frac14,\frac12],[\frac12,\frac34]\in\fD$ is the interval $[\frac14,\frac34]\not\in\fD$ which is not diadic. 
This asymmetry is used to construct, for a given $x\in T$, a so-called domain tree and range tree with vertices in $\fD$. Namely for the domain tree, we start with a subset $L$ of diadic intervals of $\fD$ which only intersect on their boundaries and whose union is the unit interval and where the set of boundary points contains all the breakpoints of $x$. We then represent each diadic interval with a dot, and we join two of these dots with a caret whenever the union is also a diadic interval provided that the two intervals only overlap on their boundaries. The tip of the caret is the joined diadic interval. We do this recursively until we get the unit interval. For instance, in our example, for the domain of $f_C$, we can join $d_2$ and $d_3$ with a caret to obtain the diadic interval $[1/2,1]$ and then we join this and $d_1$ with a caret to obtain the diadic interval $[0,1]$. The domain tree of $C$ is illustrated in Figure \ref{f:CDcomposition1}. To get the corresponding range tree we do the same construction for the diadic intervals $x(L):=\{x(d)\mid d\in L\}$.
\begin{figure}
\includegraphics[scale=.25]{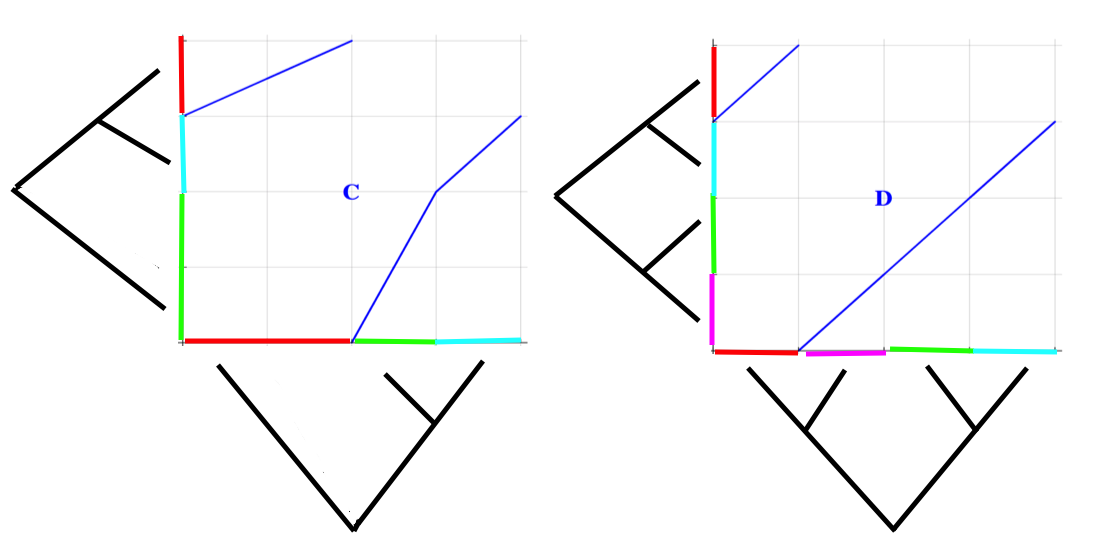}
  \caption{Construction of the doubletrees for the generators $C,D$ of the Thompson group $T$.}\label{f:CDcomposition1}
\end{figure}
\begin{figure}
\includegraphics[scale=.2]{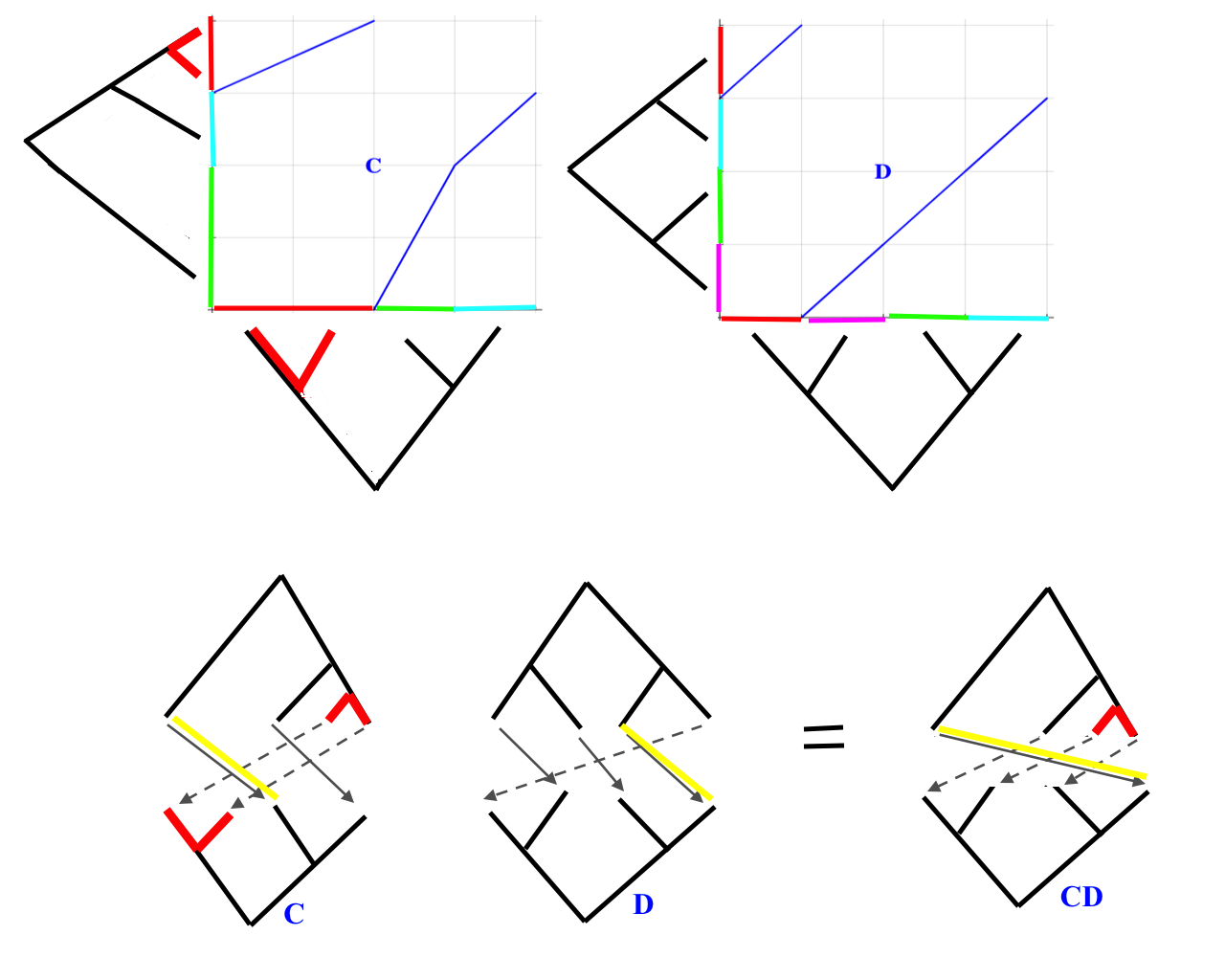}
  \caption{Doubletree composition.}\label{f:CDcomposition2}
\end{figure}

The domain tree (resp.~range tree) of $x\in T$ will be denoted with $d(x)$ (resp.~$r(x)$).
\begin{defn}[Doubletree]
The \emph{doubletree} of $x\in T$ is the pair of trees $r(x)$, $d(x)$ together with the permutation (a rotation)  
$\sigma:\mathrm{leaves}(r(x))\to \mathrm{leaves}(d(x))$
of its leaves.
\end{defn}
For instance in our example, the doubletree of $C$ is shown in Figure \ref{f:CDcomposition2}.
Composition of doubletrees is done as if we were composing its corresponding piecewise linear maps. 
However, it might be necessary to subdivide the diadic intervals in domain and range to accomplish this task. 
(This corresponds to replace  a domain leaf and its corresponding range leaf with a double caret).
For instance, if we compose  $C\circ D$ using the doubletrees shown in Figure \ref{f:CDcomposition1}, we see that the range diadic interval $[0,1/4]$ of $D$ is too small to be composed with the domain diadic interval $[0,1/2]$ of $C$. But once we subdivide  $[0,1/2]\in \mathrm{leaves}(d(C))$  into $[0,1/4]$, $[1/4,1/2]$ (and subdivide as well the image under $C$) then the composition can be done. See Figure \ref{f:CDcomposition2}.

A doubletree with no double-carets is said to be \emph{reduced}. 
\subsection{The action of $C$ and $D$ on any element of $T$}
We investigate the doubletree composition $C\circ x$, $x\in T$ in three cases, where the doubletrees of $C$ and $x$ are assumed to be reduced (i.e. no double-carets).\\
\textbf{Case 1C: (non-degenerate)}\\
Assume that the root of the range tree $r(x)$ has a left and a right node. The tree starting at the left node is called tree I. Assume that the right node has a left node and a right node. The tree starting at the right-left node is called tree II, and the one starting at the right-right node is called tree III. The trees I, II, III, can be 0-trees, i.e. they can be leaves.

Then the root of the range tree $r(C\circ x)$ has a left and a right node. The tree starting at the left node is tree II. The right node has a left node and a right node. The tree starting at the right-left node is tree III, and the one starting at the right-right node is tree I.  The leaves of tree I, which map to leaves of the domain tree $d(x)$, stay with the tree, and the same holds for the other two trees II, III. See Figure \ref{f:case1C}.
\begin{figure}
\includegraphics[scale=.4]{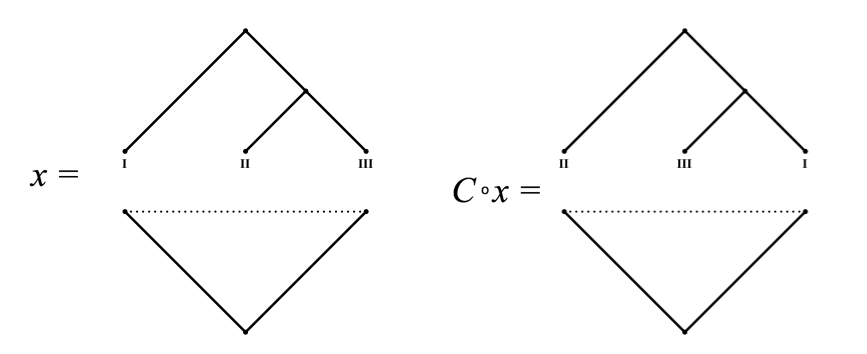}\\
  \caption{Case 1C: Trees I, II, III appear in the range tree of $x\in T$.
  The action of $C$ simply rotates these trees (and each tree keeps its own leaves). 
  }\label{f:case1C}
\end{figure}
The rotation permutation starts with $\sigma(C\circ x)(1)=\sigma(x)(\sigma(C)(1))$.

Reduction occurs only if tree III and tree I are 0-trees and the leaves $\sigma_{III}$, $\sigma_{I}$ on the domain form a caret. Reduction occurs again only if tree II is also a 0-tree; in such case we have the identity.
An example of case 1C, is shown in Figure \ref{f:case1Cexample}.\\
\begin{figure}
\includegraphics[scale=.3]{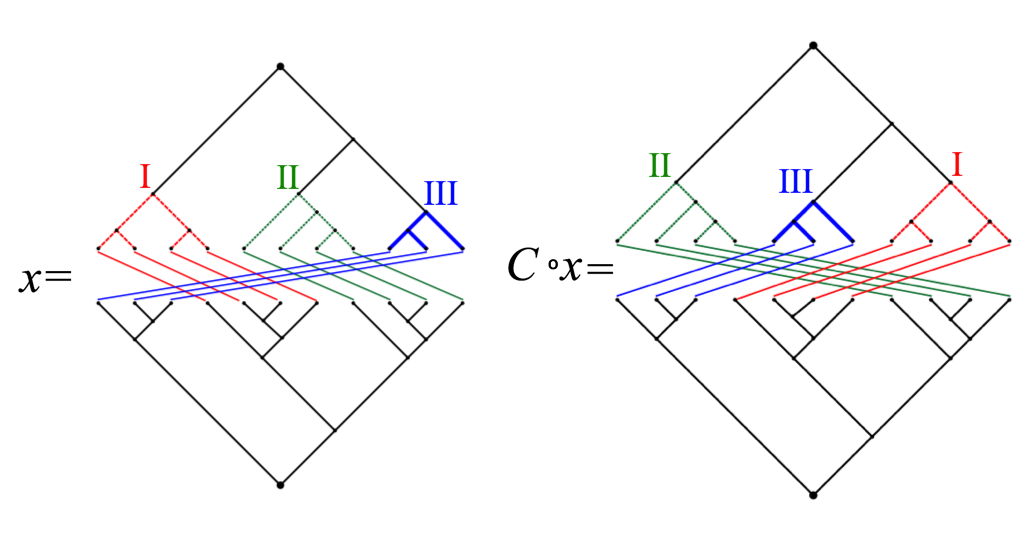}\\
  \caption{ Composition $C\circ x$, where $x\in T$ satisfies the hypothesis of case 1C (non-degenerate). 
  }\label{f:case1Cexample}
\end{figure}
\textbf{Case 2C:(degenerate, missing tree II)}\\
Assume that the root of the range tree $r(x)$ has a left and a right node. Assume that the right node is a 0-tree (a leaf), which maps to another leaf of the domain tree $d(x)$. Insert a double caret at these two leaves and proceed as in case 1C. An example for this case is shown in Figure \ref{f:case2Cexample}.\\
\begin{figure}
\includegraphics[scale=.3]{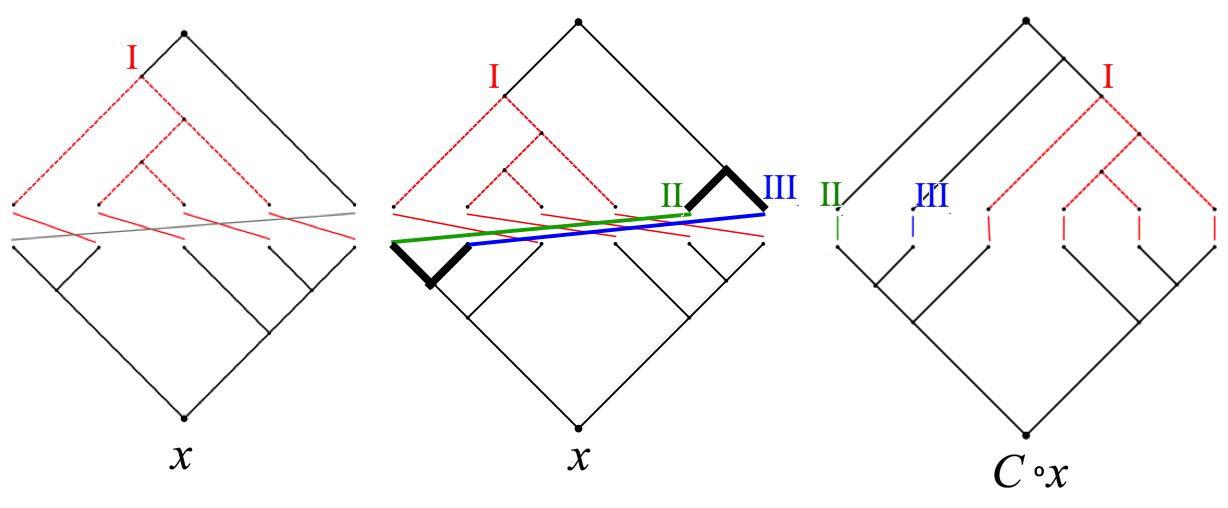}\\
  \caption{ Composition $C\circ x$, where $x\in T$ satisfies the hypothesis of case 2C (degenerate). 
  }\label{f:case2Cexample}
\end{figure}
\textbf{Case 3C:(degenerate, missing trees, I,II,III)}\\
The element $x$ is the identity. Thus $C\circ x=C$.\\\\

The composition $D\circ x$, $x\in T$ is done in 5 cases, where the trees of $D$ and $x$ are assumed to be reduced.\\
\textbf{Case 1D: (non-degenerate)}\\
Assume that the root of the range tree $r(x)$ has a left and a right node, which both have also a left node and a right node.
The tree starting at the left-left node is called tree I. The one starting at the left-right node is called tree II. The one starting at the right-left node is called tree III, and the one starting at the right-right node is called tree IV. The trees I, II, III, IV can be 0-trees, i.e. they can be leaves.

Then, the root of the range tree $r(D\circ x)$ has a left and a right node, which both have also left and right nodes.
The tree starting at the left-left node is tree II. The tree starting at the left-right node is tree III.
The tree starting at the right-left node is tree IV. The tree starting at the right-right node is tree I.
The leaves of tree I, which map to leaves of the domain tree $d(x)$, stay with the tree, and the same holds for the other three trees II, III, IV. See Figure \ref{f:case1D}.
The rotation permutation starts with $\sigma(D\circ x)(1)=\sigma(x)(\sigma(D)(1))$.
\begin{figure}
\includegraphics[scale=.35]{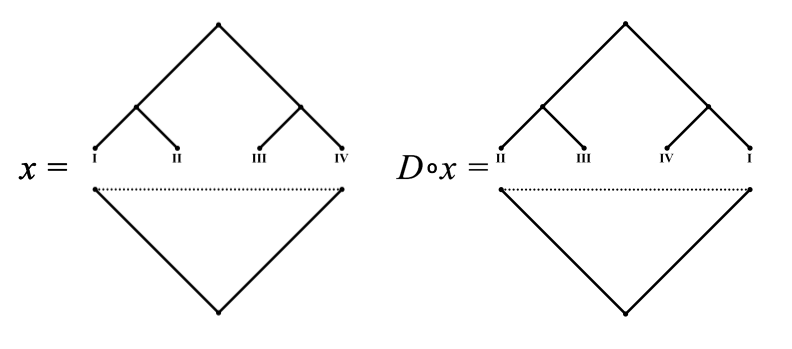}\\
  \caption{Case 1D: Trees I, II, III, IV appear in the range tree of $x\in T$.
  The action of $D$ simply rotates these trees (and each tree keeps its own leaves). 
  }\label{f:case1D}
\end{figure}

Reduction occurs only in the following two cases:
\begin{itemize}
\item [($a$)] II and III are 0-trees and the leaves $\sigma_{\mathrm{II}},\sigma_{\mathrm{III}}$ on the domain form a caret.
\item [($b$)] IV and I are 0-trees and the leaves $\sigma_{\mathrm{IV}},\sigma_{\mathrm{I}}$ on the domain form a caret.
\end{itemize}
Reduction occurs again, only if both ($a$) and ($b$) occur; in such case we have the identity.
An example of case 1D with reduction ($a$) is shown in Figure \ref{f:case1Dexample}.\\
\begin{figure}
\includegraphics[scale=.3]{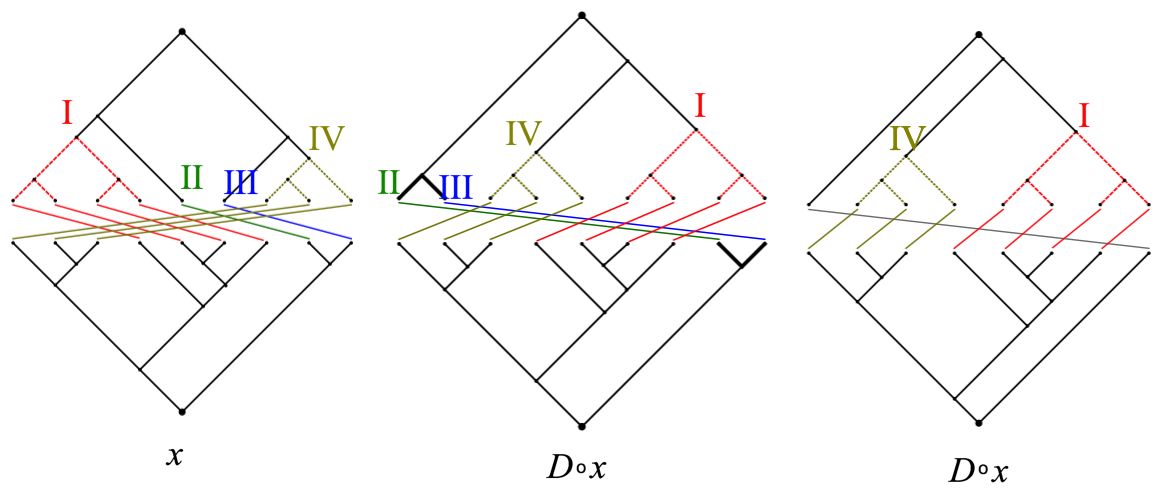}\\
  \caption{ Composition $D\circ x$, where $x\in T$ satisfies the hypothesis of case 1D (non-degenerate). 
  }\label{f:case1Dexample}
\end{figure}
\textbf{Case 2D: (degenerate, missing tree II)}\\
Assume that the root of the range tree $r(x)$ has a left and a right node. Assume that the left node is a leaf which we call $v$, while the right node has a left and a right node.  The leaf $v$ maps to a leaf of the domain tree $d(x)$. Insert a double caret at these two leaves. Then proceed as in case 1D. An example is shown in Figure \ref{f:case2Dexample}.\\
\begin{figure}
\includegraphics[scale=.3]{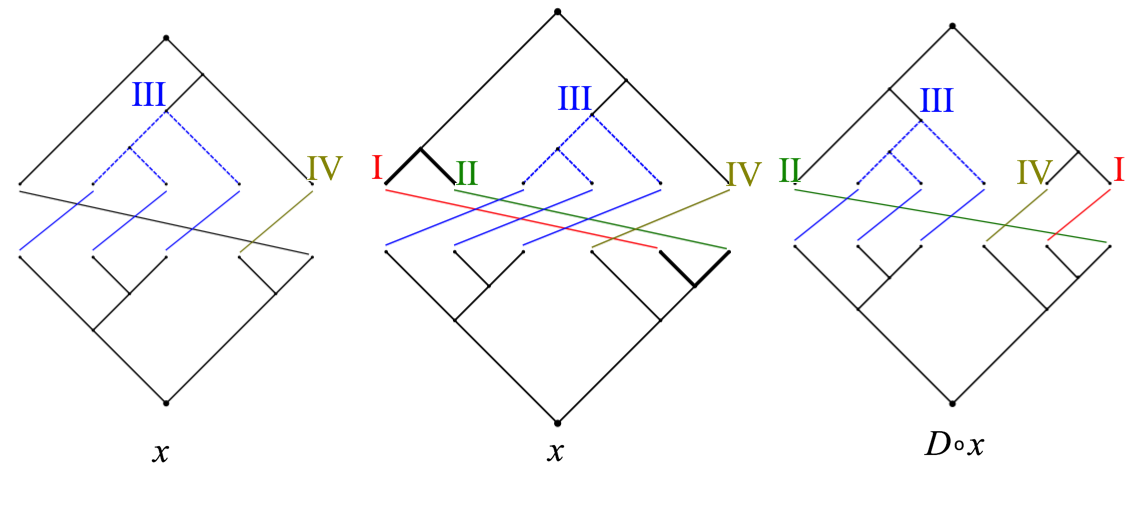}\\
  \caption{ Composition $D\circ x$, where $x\in T$ satisfies the hypothesis of case 2D (degenerate). 
  }\label{f:case2Dexample}
\end{figure}
\textbf{Case 3D: (degenerate, missing tree III)}\\
Assume that the root of the range tree $r(x)$ has a left and a right node. Assume that the right node is a leaf which we call $v$, while the left node has a left and a right node.  The leaf $v$ maps to a leaf of the domain tree $d(x)$. Insert a double caret at these two leaves. Then proceed as in case 1D. An example is shown in Figure 
\ref{f:case3Dexample}.\\
\begin{figure}
\includegraphics[scale=.3]{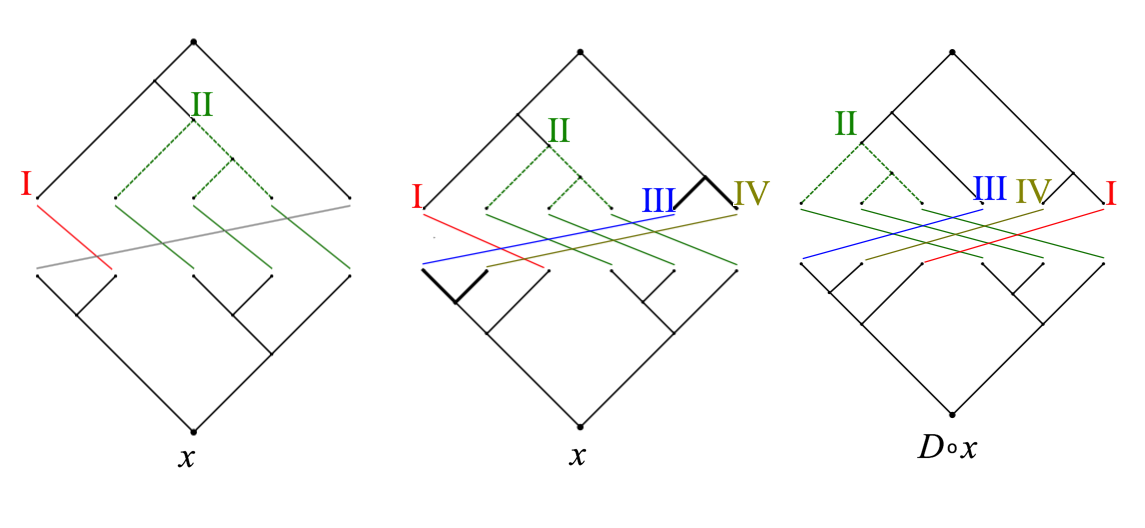}\\
  \caption{ Composition $D\circ x$, where $x\in T$ satisfies the hypothesis of case 3D (degenerate). 
  }\label{f:case3Dexample}
\end{figure}
\textbf{Case 4D: (degenerate, missing trees II, III)}\\
Assume that the root of the range tree $r(x)$ has a left and a right node, such that both trees are 0-trees.
Since $x$ is assumed to be reduced, $x=D^2$. Thus $D\circ x= D^3=D^{-1}$.\\
\textbf{Case 5D: (degenerate, missing trees I, II, III, IV)}\\
The element $x$ is the identity. Thus $D\circ x=D$.\\

\subsection{The cyclic reduced numbers $\zeta_n$}
Let $R:T\to T$, be defined by
$$R(x)(t) = 1-x(1-t), \quad   x\in T,\, t\in\R/\Z.$$ 
 Then the graph of $R(x)$ is the graph of $x$ rotated by 180 degrees. The rotation map $R$ is a group isomorphism, and
 \begin{equation}\label{e:CinvDinvReflector}
 R\circ R= id,\qquad C^{-1}=D^2 R(C) D^2, \qquad D^{-1}=D^2R(D) D^2.
 \end{equation}
 The doubletree of $R(x)$, $x\in T$, is obtained by swapping all the nodes of every caret in the domain and range trees of $x$, i.e. left node (resp.~right node) becomes right node (resp.~left node). An example is shown in Figure 
\ref{f:Rx-example}.\\
\begin{figure}
\includegraphics[scale=.4]{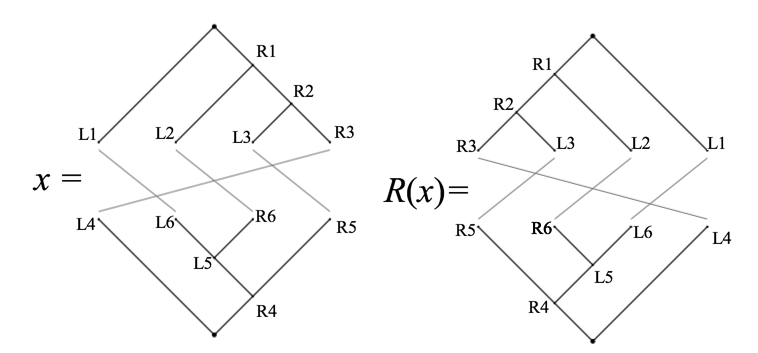}\\
  \caption{ The rotation operator $R$ on doubletrees swaps the nodes of every caret in the domain and range trees.  
  }\label{f:Rx-example}
\end{figure}
Define
$$a:=C+C^2\qquad b:=D+D^2+D^3,$$
and recall that $C^3=I$, $D^4=I$. Since $C, D\in L(T)$ are unitaries, $C^*=C^{-1}$ and $D^*=D^{-1}$. Moreover, $a$ and $b$ are self-adjoint, i.e. $a^*=a$ and $b^*=b$.
Define the ``cyclic reduced" numbers
$$\zeta_n:=\tau((ab)^n),\quad n\in\N.$$
Define $N_0:=\{D,D^2,D^{-1}\}$. For $n\in\N$ let,
\begin{eqnarray*}
N_n&:=&\{(t_0,s_1,t_1,\ldots,s_n,t_{n})\mid s_j\in\{C,C^{-1}\},\,j=1,\ldots,n,\,\,\\
&& \qquad\qquad\qquad \qquad \qquad \,\,t_i\in\{D,D^2,D^{-1}\},\,i=0,\ldots,n\},
\end{eqnarray*}
and note that 
$$ b(ab)^n=\sum_{(t_0,s_1,t_1,\cdots s_n,t_{n})\in N_n} t_0(s_1t_1\cdots s_nt_{n}),\qquad n\in\N_0.$$
We define the product of a tuple as follows
\begin{equation}\label{e:letterproduct}
\pi((t_1,t_2,\ldots, t_n)):=t_1t_2\cdots t_{n},\qquad\quad t_1,\ldots,t_n\in\{C,D,C^{-1},D^2,D^{-1}\},
\end{equation}
and the inverse of each entry of a tuple as
\begin{equation}\label{e:letterinverse}
J((t_1,t_2,\ldots, t_{n})):=(t_1^{-1},t_2^{-1},\ldots,t_n^{-1}).
\end{equation}
For $w=(t_1,\ldots, t_{n})$ it holds by Eq.~(\ref{e:CinvDinvReflector})
\begin{equation}\label{e:ri_EQ_D2_R_D2}
\pi(J(w))=D^2 R(\pi(w)) D^2.
\end{equation}
We say that the word $\pi(J(w))$ is the ``reverse-inverse" of the word
$\pi(w)$. For example, if $w=(C,D,C^2)$ then $\pi(J(w))=C^{-1}D^{-1}C^{-2}$, which corresponds to first reversing the letters of the word $CDC^2$ and then taking the inverse. 
Define for $n\in\N$,
\begin{eqnarray*}
S_n&:=&\{w, J(w')\mid w,w'\in \{C\}\times N_{n-1} ,\,\, \pi(w)<\pi(J(w)),\,\, \pi(J(w'))<\pi(w')\}\\
S'_n&:=&\{w, J(w')\mid w,w'\in N_{n-1}\times\{ C^{-1}\} ,\,\, \pi(w)<\pi(J(w)),\,\,\pi(J(w'))<\pi(w')\}\\
E_n&:=&\{ w\in \{C \}\times N_{n-1} \mid \pi(w)=\pi(J(w))\}.\\
E'_n&:=&\{ w\in N_{n-1}\times \{C^{-1}\} \mid \pi(w)=\pi(J(w))\}.
\end{eqnarray*}
Here the comparison $\pi(w)<\pi(J(w))$ is done by comparing the lexicographic order of their corresponding serialized (e.g. preorder) reduced doubletrees.
The letter $S_n$ stands for smaller, e.g. that the word $\pi(w)$ is smaller than its reverse-inverse $\pi(J(w))$.
Similarly, the letter $E_n$ stands for equal.
 Since $b(ab)^n$ is self-adjoint,  $w\in \{C\}\times N_n$ if and only if $w^{-1}\in N_n\times \{C^{-1}\}$, where $w^{-1}$ is the inverse of each entry in the tuple $w$ and in reverse order.
Define for $n\in\N$,
\begin{equation}\label{e:zeta_se_even}
\zeta_{2n}^{s}:=\sum_{(x,y)\in S_n\times S'_n} \langle x,y\rangle,\qquad \zeta_{2n}^{e}:=\sum_{(x,y)\in E_n\times E'_n} \langle x,y\rangle
\end{equation}
\begin{equation}\label{e:zeta_se_odd}
\zeta_{2n+1}^{s}:=\sum_{(x,y)\in S_{n}\times S_{n+1}'} \langle x,y\rangle,\qquad \zeta_{2n+1}^{e}:=\sum_{(x,y)\in E_{n}\times E'_{n+1}} \langle x,y\rangle,
\end{equation}
where $\langle x,y\rangle:=\tau(y^*x)$ is the inner product on $L(T)$ associated with the trace $\tau$. 
Then
\begin{lem}\label{l:zeta_se_zetarel}
$$
\zeta_{n}=2\zeta_{n}^s+4 \zeta_{n}^e,\quad n\ge2.\\
$$
\end{lem}
\begin{proof}
For $n\in\N$, define
$$X_n:=\{(s_1,t_1,\cdots s_n,t_n)\mid s_i\in\{C,C^{-1}\}, t_i\in\{D,D^2,D^{-1}\},i=1,\ldots,n\}.$$
Then
$$(ab)^n=\sum_{x\in X_n}\pi(x).$$
Since $(ab)^n=Cb(ab)^{n-1}+C^{-1}b(ab)^{n-1}$, we can partition $X_n$ as follows
\begin{eqnarray*}
X_n&=&\{w,J(w)\mid w\in \{C\}\times N_{n-1}\}\\ 
 &=&S_n\sqcup L_n\sqcup E_n\sqcup J(E_n), 
\end{eqnarray*}
where 
$$L_n:=\{w, J(w')\mid w,w'\in \{C\}\times N_{n-1},\,\, \pi(w)>\pi(J(w)), \,\,\pi(J(w'))>\pi(w')\}.$$
Thus 
$$(ab)^n=\sum_{x\in S_n} \pi(x)+\sum_{x\in L_n}\pi(x)+2\sum_{x\in E_n} \pi(x).$$
Similarly,
$$(ba)^n=\sum_{x\in S'_n} \pi(x)+\sum_{x\in L'_n}\pi(x)+2\sum_{x\in E'_n} \pi(x),$$
where
$$L'_n:=\{w, J(w')\mid w,w'\in N_{n-1}\times \{C^{-1}\} ,\,\, \pi(w)>\pi(J(w)), \,\,\pi(J(w'))>\pi(w')\}.$$
The inner product  $\langle \sum_{S_n} \pi(x),  \sum_{E'_n}\pi(x)\rangle =0$ because if 
there is a $w\in S_n, w'\in E'_n$,  such that $\pi(w)=\pi(w')$ then by Eq.~(\ref{e:ri_EQ_D2_R_D2}), $\pi(J(w))=\pi(J(w'))$, and thus
$\pi(w)< \pi(J(w))=\pi(J(w'))=\pi(w')=\pi(w)$, a contradiction.

Similarly, the inner product $\langle \sum_{S_n} \pi(x),  \sum_{L'_n}\pi(x)\rangle =0$ because if 
there is a $w\in S_n, w'\in L'_n$,  such that $\pi(w)=\pi(w')$ then by Eq.~(\ref{e:ri_EQ_D2_R_D2})
$\pi(w)< \pi(J(w))=\pi(J(w'))<\pi(w')=\pi(w)$, a contradiction.
The innerproducts
$$\langle \sum_{x\in S_n} \pi(x),\sum_{x\in S'_n} \pi(x)\rangle=
\langle \sum_{x\in L_n}\pi(x),\sum_{x\in L'_n}\pi(x)\rangle$$
coincide because $L_n=J(S_n)$, $L'_n=J(S'_n)$ and because by Eq.~(\ref{e:ri_EQ_D2_R_D2}) 
$$\pi(x)=\pi(x')\iff \pi(J(x))=\pi(J(x')).$$
It follows that for $n\in\N$,
\begin{eqnarray*}
\zeta_{2n}&=&\langle (ab)^n, (ba)^n\rangle\\
&=&\langle \sum_{x\in S_n} \pi(x),\sum_{x\in S'_n} \pi(x)\rangle+
\langle \sum_{x\in L_n}\pi(x),\sum_{x\in L'_n}\pi(x)\rangle+
4\langle \sum_{x\in E_n} \pi(x),\sum_{x\in E'_n} \pi(x)\rangle\\
&=&2\langle \sum_{x\in S_n} \pi(x),\sum_{x\in S'_n} \pi(x)\rangle+
4\langle \sum_{x\in E_n} \pi(x),\sum_{x\in E'_n} \pi(x)\rangle\\
&=&2\zeta_{2n}^s+4 \zeta_{2n}^e,
\end{eqnarray*}
and
\begin{eqnarray*}
\zeta_{2n+1}&=&\langle (ab)^n, (ba)^{n+1}\rangle\\
&=&2\langle \sum_{x\in S_n} \pi(x),\sum_{x\in S'_{n+1}} \pi(x)\rangle+
4\langle \sum_{x\in E_n} \pi(x),\sum_{x\in E'_{n+1}} \pi(x)\rangle\\
&=&2\zeta_{2n+1}^s+4 \zeta_{2n+1}^e.
\end{eqnarray*}
\end{proof}
$ $\\

\subsection{Amenability}
\begin{defn}[Amenability]
A discrete group $\Gamma$ is said to be amenable if there exists a finitely additive measure 
$\mu:P(\Gamma)\to[0,1]$ with $\mu(\Gamma)=1$ such that
$$\mu(xA)=\mu(A),\quad \text{for all } A\in P(\Gamma),\,\,  x\in\Gamma,$$
where $P(\Gamma)$ denotes the power set of $\Gamma$.
\end{defn}
$ $\\
\noindent
Let $P,Q,R\in L(T)$ be the projections given by
\begin{equation}\label{e:PQR}
P:=\frac{I+S}{2}\qquad Q:=\frac{I+C+C^2}{3}\qquad R:=\frac{I+D+D^2+D^3}{4},
\end{equation}
where $S:=D^2$.

\begin{thm}\label{t:QRleQP}
Let $P,Q,R\in L(T)$ be the projections in Eq.~(\ref{e:PQR}), which are defined in terms of the generators $C,D$ of Thompson group $T$.
Then
$$||QR||\le \frac{1+\sqrt{2}}{\sqrt{6}}.$$
Moreover, the upper bound is attained if the Thompson group $F$ is amenable.
\end{thm}

\begin{proof}
It is well-known that
$$T\supset\langle S,C\rangle = \Z_2\star \Z_3\cong PSL(2,\Z)=SL(2,\Z)/\{\pm I\},$$
(cf.~\cite[p.11,~Example 1.5.3]{JPS}).
By \cite[Proposition~2.5.9]{BrownOzawa} the reduced $C^*$-algebra of any subgroup $H$ of a (discrete) group $G$ is a subalgebra of  $C_r^*(G)$ in a canonical way, namely, one takes the left regular representation of $G$ and restricts it to $H$.
Hence
$$C_r^*(\Z_2\star \Z_3)=C_r^*(\langle S,C\rangle)\subset C_r^*(T).$$
By this and Example \ref{ex:freeCase} (or \cite[Remark 15]{ABHaagerup}) we get
$$||PQ||=\frac{1+\sqrt2}{\sqrt6}.$$   
By spectral reasons
\begin{equation}\label{e:R_le_P}
R\le P.
\end{equation}
Hence $0\le QRQ\le QPQ$, and by the $C^*$-identity (i.e. $||x^*x||=||x||^2$)  we get
$$||QR||^2=||QRQ||\le ||QPQ||.$$
Since the norm of an element is the same as the norm of its adjoint (i.e. $||x^*||=||x||$),
$$||QPQ||^{1/2}=||PQP||^{1/2}=||PQ||=\frac{1+\sqrt 2}{\sqrt{6}}.$$
Thus
\begin{equation}
  ||QR||\le \frac{1+\sqrt 2}{\sqrt{6}}=0.97\ldots.
\end{equation}
We will now prove that  $||QR|| \geq \frac{1+\sqrt{2}}{\sqrt{6}}$ if $F$ is amenable.
Let $H:=\ell^2(\Z[\frac12]\cap [0,1))$, and let $\pi:T\to B(H)$ be the representation given by
$$\pi(x)\delta_t=\delta_{x(t)},\qquad x\in T,\,\, t\in \Z[1/2]\cap[0,1),$$
where the set $\{\delta_t\mid t\in   \Z[1/2]\cap[0,1)\}$ is an orthonormal basis for the Hilbert space $H$.
Let $S_1,S_2\in B(H)$ be the isometries defined by
$$S_1\delta_t:=\delta_{t/2},\qquad S_2\delta_t:=\delta_{(t+1)/2},\qquad t\in \Z[1/2]\cap [0,1),$$
and note that they satisfy the relation $S_1 S_1^*+S_2 S_2^*=I$.
Then by \cite[Proposition 4.3]{HO}$, \pi(\C T)$ is contained in the Cuntz algebra $\fO_2=C^*(S_1,S_2)$. 
If $F$ is amenable then by \cite[Proposition 4.4]{HO}, $\pi$ is weakly contained in the left regular representation of $T$.
This is equivalent to the existence of a *-homomorphism $\rho$ from the reduced $C^*$-algebra $C_r^*(T)$ into the $C^*$-algebra generated by $\pi(T)$ such that $\pi=\rho\circ \lambda$.
Hence
$$||\pi(x)||_{B(H)}=||\rho(\lambda(x))||_{B(H)}\le ||\lambda(x)||_{C_r^*(T)},\quad x\in \C T.$$
Thus amenability of $F$ implies that the norm
$||QR||$ is greater or equal than the norm $||\pi(QR)||$. Thus, it suffices to show that 
$||\pi(QR)|| = \frac{1+\sqrt{2}}{\sqrt{6}}$. We have $||\pi(QR)||^2 = ||\pi(RQR)||$,
and 
$$ \pi(RQR) = \frac1{48} (S_1 + S_2)^2 (2I + S_1 + S_2)^* (2I + S_1 + S_2) (S_1 + S_2)^{* 2}. $$ 
Substituting $x := \frac{1}{\sqrt{2}}(S_1 + S_2)$ gives
$$ \pi(RQR) = \frac16 x^2 (\sqrt{2}I + x)^* (\sqrt{2}I + x) x^{*2}. $$ 
Since $x$ is a proper isometry (i.e. $x^*x=I\ne xx^*$),
$$ ||x^2 (\sqrt{2}I + x)^* (\sqrt{2}I + x) x^{*2}|| = ||(\sqrt{2}I + x)^* (\sqrt{2}I + x)||. $$
Consequently, 
$$ ||\pi(RQR)|| = \frac16 ||(\sqrt{2}I + x)^* (\sqrt{2}I + x)|| 
            = \frac16 ||\sqrt{2}I + x||^2. $$
            
Now, the spectrum of $x$ is the entire unit disc because $x$ is a proper isometry. (One uses the Wold decomposition to show this).
Thus the spectral radius of $\sqrt{2}I + x$ is $\sqrt{2} +1$. Hence the norm is at least this number and by the triangle inequality we get
$$ ||\sqrt{2}I + x|| = 1 + \sqrt{2}, $$ 
as required.
\end{proof}

\section{\textbf{Two projections}}
Let $\Gamma:=\langle c,d\rangle$ be a discrete group on two generators such that $c$ is of order $k$ and $d$ of order $\ell$, for some integers $\ell\ge k\ge 2$.
Let $L(\Gamma)$ denote its von Neumann algebra, i.e. $L(\Gamma)$ is the von Neumann algebra in $B(\ell^2(\Gamma))$ generated by $\Gamma$. Then
$$\tau(h) = \langle h\delta_e,\delta_e\rangle,\qquad h \in L(\Gamma)$$
defines a normal faithful tracial state on $L(\Gamma)$ (see e.g. [23, Section 6.7]). 
Let 
$$P:=\frac{e+c+\cdots+c^{k-1}}{k}\in\C\Gamma\subset C_r^*(\Gamma)\subset L(\Gamma)\subset B(\ell^2(\Gamma))$$
$$Q:=\frac{1+d+\cdots+d^{\ell-1}}{\ell}\in\C\Gamma\subset C_r^*(\Gamma)\subset L(\Gamma)\subset B(\ell^2(\Gamma))$$
 be fixed projections such that 
$$||PQ||<1.$$ 
Let 
$$\alpha:=\tau(P)=\frac1k,\qquad \beta:=\tau(Q)=\frac1{\ell}$$
and notice that $$ 0< \beta \le \alpha\le \frac12.$$
It follows by P. R. Halmos' work \cite{Halmos} that the $C^*$-algebra $A:=C^*(P,Q,I)$ generated by 2 projections $P,Q$ and the identity $I$ is the direct sum of an abelian $C^*$-algebra and a $C^*$-algebra of type $I_2$ i.e. of the form $C(X)\otimes M_2(\C)$ where $X$ equals the spectrum of a certain positive contraction.
The abelian part is of dimension at most 4 and its minimal projections are given as
$$ e_{00}:=I-P\lor Q=(I-P)\land(I-Q)$$
$$e_{11}:=P\land Q\quad\qquad e_{10}:=P\land (I-Q)\quad\qquad e_{01}:=(I-P)\land Q,$$
where some, and even all may vanish.
The central support for the $I_2$ part of $A$ is denoted $f$ and it is given by
\begin{equation}\label{e:f}
f=I-e_{11}-e_{10}-e_{01}-e_{00}.
\end{equation}
We define
\begin{equation}\label{e:R-S}
R:=P-e_{11}-e_{10},\qquad S:=Q-e_{11}-e_{01}
\end{equation}
and see that $R$ and $S$ are projections satisfying $R\le f$, $S\le f$ because $Pe_{1i}=e_{1i}P=e_{1i}$, $Pe_{0i}=e_{0i}P=0$ and $Qe_{i1}=e_{i1}Q=e_{i1}$, $Qe_{i0}=e_{i0}Q=0$, $i=0,1$. Moreover,
$$R\land S=0\qquad R\lor S=f.$$
So inside $fAf$, the pair of projections $R,S$ actually fit the description given by Halmos.
Thus, there exists an isomorphism $\Phi$ of $f Af$ onto
$C(\sigma(RSR))\otimes M_2(\C)=M_2(C(\sigma(RSR)))$, such that
$$\Phi(f)=I_{C(\sigma(RSR))}\otimes \left(\begin{array}{cc}
  1  & 0   \\
  0 &  1   
\end{array}
\right)$$
$$\Phi(R)=I_{C(\sigma(RSR))}\otimes \left(\begin{array}{cc}
  1  & 0   \\
  0 &  0   
\end{array}
\right)$$
and
\begin{eqnarray*}
\Phi(S)&=&\left(\begin{array}{cc}
  RSR  & (RSR)^{1/2}(f-RSR)^{1/2}   \\
  (RSR)^{1/2}(f-RSR)^{1/2} &  f-RSR   
\end{array}
\right).
\end{eqnarray*}
Now, 
$$P-\alpha I= (R-\alpha f)+(1-\alpha)e_{11}+(1-\alpha)e_{10}-\alpha e_{01}-\alpha e_{00}$$
$$Q-\beta I= (S-\beta f)+(1-\beta)e_{11}+(1-\beta)e_{01}-\beta e_{10}-\beta e_{00}$$
$$(P-\alpha I)(Q-\beta I)=(R-\alpha f)(S-\beta f) + (1-\alpha)(1-\beta) e_{11}-\alpha(1-\beta)e_{01}-\beta(1-\alpha)e_{10}+$$
$$(-\alpha)(-\beta) e_{00}.$$
Thus since $R e_{ij}=S e_{ij}=f e_{ij}=0$ for $i,j\in\{0,1\}$ we get
\begin{eqnarray}\label{e:tauEFn}
\qquad\qquad\tau\Big(\big((P-\alpha I)(Q-\beta I)\big)^n\Big)&=&\tau\Big(\big((R-\alpha f)(S-\beta f)\big)^n\Big)+(\alpha\beta)^n\tau(e_{00})+\\
\nonumber&&(-\alpha (1-\beta))^n \tau(e_{01})+(-\beta(1-\alpha ))^n \tau(e_{10})+\\
\nonumber&&((1-\alpha )(1-\beta))^n \tau(e_{11}).
\end{eqnarray}
Similarly, since
$$PQ=RS+e_{11}$$
$$(I-P)(I-Q)=(f-R)(f-S)+e_{00}$$
$$P(I-Q)=R(f-S)+e_{10}$$
$$(I-P)Q=(f-R)S+e_{01}$$
we get
$$\tau((PQ)^n)=\tau((RS)^n)+\tau(e_{11})$$
$$\tau\Big(\big((I-P)(I-Q)\big)^n\Big)=\tau\Big(\big((f-R)(f-S)\big)^n\Big)+\tau(e_{00})$$
$$\tau\Big(\big(P(I-Q)\big)^n\Big)=\tau\Big(\big(R(f-S)\big)^n\Big)+\tau(e_{10})$$
$$\tau\Big(\big((I-P)Q\big)^n\Big)=\tau\Big(\big((f-R)S\big)^n\Big)+\tau(e_{01}).$$
Since $||PQ||<1$, $||e_{11}||=||e_{11}PQ||<1$ because the norm is submultiplicative, and thus the minimal projection $$e_{11}=0,$$ and since $PQP=RSR+e_{11}$, the spectrum $$\sigma(RSR)=\sigma(PQP)\subset [0,1).$$
By the continuous functional calculus for the positive element $RSR\in fAf$ we can write
$$\Phi(f)= \left(\begin{array}{cc}
  1  & 0   \\
  0 &  1   
\end{array}
\right),
\qquad \Phi(R)= \left(\begin{array}{cc}
  1  & 0   \\
  0 &  0   
\end{array}
\right).$$
$$
\Phi(S)=
\left(\begin{array}{cc}
  t  & \sqrt{t(1-t)}   \\
  \sqrt{t(1-t)} &  1-t   
\end{array}
\right),\qquad\qquad\quad  0<t<1.
$$
\begin{eqnarray*}
\Phi(RSR)&=&
\left(\begin{array}{cc}
  t  & 0   \\
  0 &  0   
\end{array}
\right),
\qquad\quad 0<t<1.
\\
\Phi(R(f-S)R)&=&
\left(\begin{array}{cc}
  1-t  & 0   \\
  0 &  0   
\end{array}
\right),\qquad 0<t<1.
\\
\Phi((f-R)S(f-R))&=&
\left(\begin{array}{cc}
  0  & 0   \\
  0 &  1-t   
\end{array}
\right),
\qquad 0<t<1.
\\
\Phi((f-R)(f-S)(f-R))&=&
\left(\begin{array}{cc}
  0  & 0   \\
  0 &  t   
\end{array}
\right),\qquad\qquad 0<t<1.
\end{eqnarray*}

The restriction of $\tau$ to $f A f$ is a trace on this algebra, so by the spectral theorem to the positive element $RSR$, there must exist a positive measure $\mu$  on $\sigma(RSR)$ such that for any $x\in f Af$ we have
\begin{eqnarray*}
\tau(x)&=&\int_{t\in\sigma(RSR)} \frac12\Phi(x)_{11}(t) + \frac12\Phi(x)_{22}(t)d\mu(t)\\
&=&\int_0^1\tau_t(x)d\mu(t)
\end{eqnarray*}
where 
$$\tau_t(x):=\frac12\Tr(\Phi(x)(t)),$$
 and $\frac12\Tr$ is the normalized trace of $M_2(\C)$.
Thus  
 $$1=\tau(I)= \int_0^1\tau_t(f)d\mu(t)+\tau(e_{11})+\tau(e_{10})+\tau(e_{01})+\tau(e_{00}),$$ 
and since $\tau((pq)^n)=\tau((pqp)^n)$ for any two projections $p,q$ and any $n\in\N$, we get
\begin{eqnarray}\label{e:tauPQvna}
\tau\big((PQ)^n\big)&=&\int_0^1\tau_t\big((RS)^n\big)d\mu(t)\,\,+\,\, \tau(e_{11})\\
\nonumber\tau\Big(\big((I-P)(I-Q)\big)^n\Big)&=&\int_0^1\tau_t\Big(\big((f-R)(f-S)\big)^n\Big)d\mu(t)\,\,+\,\, \tau(e_{00})\\
\nonumber\tau\Big(\big(P(I-Q)\big)^n\Big)&=&\int_0^1\tau_t\Big(\big(R(f-S)\big)^n\Big)d\mu(t)\,\,+\,\, \tau(e_{10})\\
\nonumber\tau\Big(\big((I-P)Q\big)^n\Big)&=&\int_0^1\tau_t\Big(\big((f-R)S\big)^n\Big)d\mu(t)\,\,+\,\, \tau(e_{01}).
\nonumber\end{eqnarray}
It is not difficult to show that $\mu(\{0\})=0$.
We extend $\mu$ to a probability measure on $[0,1]$ by putting
\begin{eqnarray*}
&&\mu(\{0\}):=\mu_{10}+\mu_{01}\\
&&\mu(\{1\}):=\mu_{00}+\mu_{11},
\end{eqnarray*}
where
$$\mu_{ij}:=\tau(e_{ij}),\quad i,j\in\{0,1\}.$$
\begin{lem}\label{l:nu_mn}
Define $$\nu:=\mu-(\mu_{10}-\mu_{01})\delta_0-(\mu_{00}-\mu_{11})\delta_1.$$
Then $\nu$ is a measure on $[0,1]$ with total mass $\nu([0,1])=2\beta$, and
$$\tau((PQ)^n)=\frac12\int_{[0,1]}t^n d\nu(t), \qquad n\ge 1.$$
Moreover,
$$\mu_{10}-\mu_{01}=\alpha-\beta\ge0, \qquad\mu_{11}=0,\qquad \mu_{00}=1-\alpha-\beta\ge0, \qquad\nu(\{1\})=0.$$

\end{lem}

\begin{proof}
Letting $n=1$ in Eq.~(\ref{e:tauPQvna}), we get
\begin{eqnarray*}
1=\tau(I)&=&\mu_{00}+\mu_{01}+\mu_{10}+\mu_{11}+\mu((0,1)).\\
\alpha=\tau(P)&=&\tau(P(I-Q))+\tau(PQ)=\mu_{10}+\mu_{11}+\frac12\mu((0,1)).\\
\beta=\tau(Q)&=&\tau((I-P)Q)+\tau(PQ)=\mu_{01}+\mu_{11}+\frac12\mu((0,1)).\\
1-\alpha=\tau(I-P)&=&\mu_{00}+\mu_{01}+\frac12\mu((0,1)).\\
1-\beta=\tau(I-Q)&=&\mu_{00}+\mu_{10}+\frac12\mu((0,1)).\\
\end{eqnarray*}
Combining these equations we get
$$0\le \alpha-\beta=\tau(P)-\tau(Q)=\mu_{10}-\mu_{01}$$
$$0\le1-\alpha-\beta=\tau(I-P)-\tau(Q)=\mu_{00}-\mu_{11}.$$
Since $||PQ||<1$, $\mu_{11}=0$. Thus
$$ \mu_{00}=1-\alpha-\beta.$$
We can write $\mu_{ij}$, $i=0,1$ as follows
\begin{eqnarray*}
&&\mu_{01}=\frac12\mu(\{0\})-\frac{\mu_{10}-\mu_{01}}{2}\\
&&\mu_{10}=\frac12\mu(\{0\})+\frac{\mu_{10}-\mu_{01}}{2}\\
&&\mu_{00}=\frac12\mu(\{1\})+\frac{\mu_{00}-\mu_{11}}{2}\\
&&\mu_{11}=\frac12\mu(\{1\})-\frac{\mu_{00}-\mu_{11}}{2},
\end{eqnarray*}
and the measure $\nu$ on $[0,1]$ defined in the lemma 
satisfies
\begin{eqnarray}\label{e:mu10relation}
\nonumber&&\nu(\{0\})=\mu(\{0\})-(\mu_{10}-\mu_{01})\\
\nonumber&&\nu(\{1\})=\mu(\{1\})-(\mu_{00}-\mu_{11})\\
\nonumber&&\mu_{01}=\frac{\nu(\{0\})+(\mu_{10}-\mu_{01})}{2}-\frac{\mu_{10}-\mu_{01}}2=\frac12\nu(\{0\})\\
&&\mu_{10}=\frac12\nu(\{0\})+(\mu_{10}-\mu_{01})\\
\nonumber&&\mu_{00}=\frac{\nu(\{1\})+(\mu_{00}-\mu_{11})}{2}+\frac{\mu_{00}-\mu_{11}}{2}=\frac12\nu(\{1\})+(\mu_{00}-\mu_{11})\\
\nonumber&&\mu_{11}=\frac12\nu(\{1\}).
\end{eqnarray}
Moreover, the total mass of $\nu$ is
\begin{eqnarray*}
\nu([0,1])&=&\nu(\{0\})+\nu(\{1\})+\mu((0,1))\\
&=&1-(\mu_{10}-\mu_{01}+\mu_{00}-\mu_{11})= 1-(\alpha-\beta+1-\alpha-\beta)\\
&=&2\beta,
\end{eqnarray*}
and
$$\int_{[0,1]}{\frac12t^n d\nu(t)}=\frac12\int_{(0,1)}t^n d\mu(t)+\frac12\nu(\{1\})=\tau((PQ)^n), \qquad n\in\N.$$
\end{proof}

By the following proposition and Lemma \ref{l:nu_mn} we can state the ``cyclic numbers"
$$\tau(((P-\alpha I)(Q-\beta I))^n),\qquad n\in\N,$$
in terms of the moments $\tau((PQ)^n)$.
\begin{thm}\label{t:mn_zn}
For $n\in\N$, 
\begin{multline*}
$$\tau(((P-\alpha I)(Q-\beta I))^n)=(\alpha \beta )^n\mu_{00}+(-\beta (1-\alpha ))^n(\mu_{10}-\mu_{01})\,\,+\\
(\alpha (1-\alpha )\beta (1-\beta ))^{n/2}\int_{[0,1]}T_n(\frac{t-\alpha -\beta +2\alpha \beta }{2\sqrt{\alpha (1-\alpha )\beta (1-\beta )}})d\nu(t),
\end{multline*}
where $T_n$ is the Chebyshev polynomial of the first kind.
\end{thm}

\begin{proof}
Let
$$E:=P-\alpha I\qquad F:=Q-\beta I,\qquad E':=R-\alpha f,\qquad F':=S-\beta f$$
Then the spectrum $\sigma(E)=\{-\alpha,1-\alpha\}$, $\sigma( F)=\{-\beta,1-\beta\}$, and by Eq.(\ref{e:tauEFn}) we get
$$\tau((E F)^n)=((-\alpha )(-\beta))^n\mu_{00}+(-\alpha (1-\beta))^n\mu_{01}+(-\beta(1-\alpha ))^n\mu_{10}+((1-\alpha )(1-\beta))^n\mu_{11}+$$
$$\int_{(0,1)} \tau_t((E'F')^n)d\mu(t)$$
$$=(\alpha \beta)^n\mu_{00}+(-\beta(1-\alpha ))^n(\mu_{10}-\mu_{01})+\nu(\{0\})(\frac12(-\beta(1-\alpha ))^n+\frac12(-\alpha (1-\beta))^n)+$$
$$\frac12\nu(\{1\})((1-\alpha )(1-\beta))^n+\int_{(0,1)}\tau_t((E'F')^n)d\mu(t)$$
$$=(\alpha \beta)^n\mu_{00}+(-\beta(1-\alpha ))^n(\mu_{10}-\mu_{01})+\int_{[0,1]}\tau_t((E'F')^n)d\nu(t),$$
where we used Eq.~(\ref{e:mu10relation}) in the second equality.

To complete the proof we just need to show that 
$$\int_{[0,1]}\tau_t((E'F')^n)d\nu(t)=(\alpha (1-\alpha )\beta (1-\beta ))^{n/2}\int_{[0,1]}T_n(\frac{t-\alpha -\beta +2\alpha \beta }{2\sqrt{\alpha (1-\alpha )\beta (1-\beta )}})d\nu(t).$$
Since
\begin{eqnarray*}
\Phi(E')&=&\Phi(R-\alpha f)=\left(\begin{array}{cc}
  1-\alpha   & 0   \\
  0 &  -\alpha    
\end{array}
\right)\\
\Phi(F')&=&\Phi(S-\beta f)=\left(\begin{array}{cc}
  t-\beta   & \sqrt{t(1-t)}   \\
  \sqrt{t(1-t)} &  1-t-\beta    
\end{array}
\right)
\end{eqnarray*}
we get
$$\Phi(E'F')=\left(\begin{array}{cc}
  (1-\alpha )(t-\beta )  & (1-\alpha )\sqrt{t(1-t)}   \\
  -\alpha \sqrt{t(1-t)} &  -\alpha (1-t-\beta )   
\end{array}
\right).
$$
Since
\begin{eqnarray*}
\mathrm{Tr}(\Phi(E'F'))&=&(1-\alpha )(t-\beta )-\alpha (1-t-\beta )=t-\alpha -\beta +2\alpha \beta\\
\det(\Phi(E'F'))&=&\alpha (1-\alpha )(\beta (1-\beta ))
\end{eqnarray*}
the characteristic polynomial $\lambda^2+Tr(\Phi(E'F'))+\det(\Phi(E'F'))$ of $\Phi(E'F')$ is 
$$\lambda^2-(t-\alpha -\beta +2\alpha \beta )\lambda+\alpha (1-\alpha )(1-\beta ).$$
The eigenvalues of $\Phi(E'F')$ are
$$\lambda_1(t):=\frac{t-\alpha -\beta +2\alpha \beta +\sqrt{(t-\alpha -\beta +2\alpha \beta )^2-4\alpha (1-\alpha )\beta (1-\beta )}}{2}$$
$$\lambda_2(t):=\frac{t-\alpha -\beta +2\alpha \beta -\sqrt{(t-\alpha -\beta +2\alpha \beta )^2-4\alpha (1-\alpha )\beta (1-\beta )}}{2}.$$

If $|t-\alpha -\beta +2\alpha \beta |\le 2\sqrt{\alpha (1-\alpha )\beta (1-\beta )}$ then $\lambda_2=\overline{\lambda_1}$ and
$$\lambda_1=|\lambda_1|e^{i\theta}\qquad \lambda_2=|\lambda_2|e^{-i\theta} \qquad |\lambda_1|=|\lambda_2|=\sqrt{\alpha (1-\alpha )\beta (1-\beta )},$$
where
$$\cos\theta=\frac{\mathrm{Re\, }\lambda_1}{2|\lambda_1|}=\frac{t-\alpha -\beta +2\alpha \beta }{2\sqrt{\alpha (1-\alpha )\beta (1-\beta )}}.$$

Let $t_1,t_2$ be the roots of $(t-\alpha -\beta +2\alpha \beta )^2-4\alpha (1-\alpha )\beta (1-\beta )$
with $t_1<t_2$.
Then for $t_1< t< t_2$ it holds
\begin{eqnarray*}
\tau_t((E'F')^n)&=&\frac{\lambda_1(t)^n+\lambda_2(t)^n}{2}\\
&=&|\lambda_1|^n\frac{e^{in\theta }+e^{-in\theta }}{2}\\
&=&|\alpha (1-\alpha )\beta (1-\beta )|^{n/2}\cos(n\theta)\\
&=&(\alpha (1-\alpha )\beta (1-\beta ))^{n/2}T_n(\frac{t-\alpha -\beta +2\alpha \beta }{2\sqrt{\alpha (1-\alpha )\beta (1-\beta )}}),
\end{eqnarray*}
where $T_n$ is the Chebyshev polynomials of the first kind.
Since both sides of this equality are polynomials which agree on the interval $[t_1,t_2]$, they must be the same polynomials.
Hence
\begin{eqnarray*}
\int_{[0,1]}\tau_t((E'F')^n)d\nu(t)&=&\int_{[0,1]} \frac{\lambda_1(t)^n+\lambda_2(t)^n}{2}d\nu(t).\\
&=&\int_{[0,1]}(\alpha (1-\alpha )\beta (1-\beta ))^{n/2}T_n(\frac{t-\alpha -\beta +2\alpha \beta }{2\sqrt{\alpha (1-\alpha )\beta (1-\beta )}})\,d\nu(t).
\end{eqnarray*}
\end{proof}

 
 \section{\textbf{The free case}}\label{s:free}
Let $p\in L(\Gamma)$ be a nonzero projection different from the identity. Here projections mean self-adjoint idempotents ($p^2=p=p^*$). 
Let $\mu$  be the spectral distribution measure of $p$  with respect to the trace $\tau$ on $L(\Gamma)$. 
Then $\supp(\mu)=\sigma(p)=\{0,1\}$ because $\tau$ is faithful, and 
$$\tau(f(p))=\int_{\sigma(p)}f d{\mu}=\int_{\sigma(p)}(f(0)1_{\{0\}}+f(1)1_{\{1\}})d{\mu}=f(0)\mu(\{0\})+f(1)\mu(\{1\}),$$
for any $f\in C(\{0,1\})$.
In particular, by letting $f(t)=t$ and $f(t)=1-t$ we get
$$\mu(\{0\})=1-\tau(p),\quad \mu(\{1\})=\tau(p).$$
Moreover, the moments $m_n(p)=\tau(p^n)=\tau(p)$ for $n\in\N$.

\begin{thm}\label{t:freemeasure}
Let $\Gamma:=\Z_k\star \Z_\ell$ be the free product on two generators $c\in \Z_k$, $d\in\Z_\ell$ for some integers $k,\ell\ge 2$. Let
$$p:=\frac{e+c+\cdots+c^{k-1}}{k},\,\,q:=\frac{e+d+\cdots+d^{\ell-1}}{\ell}\in\C\Gamma$$ 
be two nonzero projections different from the identity.
Moreover, assume that  $\alpha:=\tau(p)$, $\beta :=\tau(q)$ satisfy $0< \beta \le \alpha <1$, $\alpha +\beta < 1$ and let
\begin{eqnarray*}
&&\lambda_1:=\sqrt{\alpha (1-\beta )}+\sqrt{\beta (1-\alpha )}\\
&&\lambda_2:=\sqrt{\alpha (1-\beta )}-\sqrt{\beta (1-a)}.
\end{eqnarray*}
Then the spectral distribution measure $\mu_{\mathrm{free}}$ of $\left(
\begin{array}{cc}
 0 & (pq)^* \\
 pq & 0 \\
\end{array}
\right)\in M_2(L(\Gamma))$  with respect to the trace $\tilde\tau:=\tau\otimes\tau_2$ on $M_2(L(\Gamma))$ is 
$$\mu_{\mathrm{free}}=\frac 1{2\pi}\frac{\sqrt{(\lambda_1^2-t^2)(t^2-\lambda_2^2)}}{t(1-t^2)} 1_{[-\lambda_1,-\lambda_2]\cup[\lambda_2,\lambda_1]} dt\,\,+\,\, (1-\beta )\delta_0,
$$
where $dt$ is the Lebesgue measure.
In particular,
$$||pq||=\sqrt{\alpha (1-\beta )}+\sqrt{\beta (1-\alpha )}.$$
\end{thm}

\begin{proof}
Let $\mu$ (resp.~$\nu$) be the spectral distribution measure of $p$ (resp.~$q$) with respect to the trace $\tau$ of $L(\Gamma)$.
We now compute the S-transform (cf.\cite[p.30-32]{Voiculescu} or \cite{Voiculescu87}):
\begin{eqnarray*}
\psi_{\mu}(s)&:=&\int_0^{\infty} \frac{t s}{1- t s}d\mu(s) = \int_0^{\infty} (ts +t^2s^2+t^3s^3+\cdots)d\mu(t)\\
&=&m_1(p)s +m_2(p)s^2+m_3(p) s^3+\cdots = \alpha s+\alpha s^2+\alpha s^3+\cdots\\
&=&\frac{\alpha s}{1-s}.
\end{eqnarray*}
Similarly,
$$
\psi_{\nu}(s)=\frac{\beta s}{1-s}.
$$
The formal power series $\psi_\mu$ (resp.~$\psi_\mu$) has to satisfy the equations $\chi_\mu(\psi_\mu(s))=s$, $\psi_\mu(\chi_\mu(z))=z$ and $S_\mu(z)=\frac{z+1}{z}\chi_\mu(z)$ (resp.~$\chi_\nu(\psi_\nu(s))=s$, $\psi_\nu(\chi_\nu(z))=z$ and $S_\nu(z)=\frac{z+1}{z}\chi_\nu(z)$).
Solving for $\chi_\mu,S_\mu,\chi_\nu,S_\nu$, we get
$$z=\psi_{\mu}(s) \iff (1-s)z =\alpha  s \quad \imply \quad  \chi_{\mu}(z)=s=\frac{z}{\alpha +z}$$
$$z=\psi_{\nu}(s)  \quad \imply \quad  \chi_{\nu}(z)=s=\frac{z}{\beta +z}.$$
$$S_{\mu}(z)=\frac{z+1}{z}\chi_{\mu}(z)=\frac{z+1}{\alpha +z}$$
$$S_{\nu}(z)=\frac{z+1}{z}\chi_{\nu}(z)=\frac{z+1}{\beta +z}$$
By \cite[Theorem 3.6.3]{Voiculescu}, the multiplicative free convolution is 
$$S_{\mu\boxtimes\nu}(z)= S_{\mu}(z)S_{\nu}(z)=\frac{(z+1)^2}{(\alpha +z)(\beta +z)}.$$
Thus
$$\chi_{\mu\boxtimes\nu}(z)= \frac{z}{z+1}S_{\mu\boxtimes\nu}(z)=\frac{z(z+1)}{(\alpha +z)(\beta +z)},$$
and 
$$ s=\chi_{\mu\boxtimes\nu}(z)\quad\iff\quad (\alpha +z)(\beta +z)s=z(z+1)$$
$$\imply \psi_{\mu\boxtimes\nu}(s)=z=\frac12\frac{1-(\alpha +\beta )s\pm \sqrt{d(s)}}{s-1}$$
where $d(s)$ is the discriminant 
$$d(s):=(\alpha -\beta )^2 s^2+(4\alpha \beta -2\alpha -2\beta )s+1.$$ 
Letting $s=\frac1{\lambda}$, the discriminant $d(s)$  can be written as 
$$d(\frac1{\lambda})=\frac{(\lambda-\lambda_1^2)(\lambda-\lambda_2^2)}{\lambda^2},$$
where
\begin{eqnarray*}
\lambda_1:=\sqrt{\alpha (1-\beta )}+\sqrt{\beta (1-\alpha )}\qquad\lambda_2:=\sqrt{\alpha (1-\beta )}-\sqrt{\beta (1-\alpha )}.
\end{eqnarray*}
Let $\psi:=\psi_{\mu\boxtimes\nu}$. Then the Cauchy transform $G(\lambda)$ of $\mu\boxtimes\nu$ is 
$$G(\lambda)-\frac1{\lambda}=\frac1{\lambda}\psi(\frac1{\lambda})=\frac1{2\lambda}\frac{1-(\alpha +\beta )\frac1{\lambda}+\sqrt{d(\frac1{\lambda})}}{\frac1{\lambda}-1}=
\frac{\lambda-\alpha -\beta + \sqrt{(\lambda-\lambda_1^2)(\lambda-\lambda_2^2)}}{2\lambda(1-\lambda)}.
$$
By Stieltjes inversion formula,  $\mu\boxtimes\nu$ is given by
\begin{eqnarray}\label{e:muxnu}
\mu\boxtimes\nu(x)&=&\lim_{y\to 0^+}\left(- \frac1{\pi}\im\, G(x+iy)\right)\,dx+c_0\delta_0\\
\nonumber&=&\frac1{\pi}\frac{\sqrt{(\lambda_1^2-x)(x-\lambda_2^2)}}{2x(1-x)}1_{[\lambda_2^2,\lambda_1^2]}(x)\,dx+c_0\delta_0.
\end{eqnarray}
where $c_0:=1-\beta $, is the residue of $G$ at the simple pole $0$, and $dx$ is the Lebesgue measure.
Since $0< \beta \le \alpha <1$ and $\alpha +\beta <1$, the integral
\begin{equation}\label{e:intm0}
\int_{\lambda_2^2}^{\lambda_1^2}\frac1{\pi}\frac{\sqrt{(\lambda_1^2-x)(x-\lambda_2^2)}}{2x(1-x)}\,dx=\beta ,
\end{equation}
(cf.~\cite[Example 4.3.3]{Inge}).
By \cite[Remarks 3.6.2(iii)]{Voiculescu}, we get that $\mu\boxtimes\nu=\mu_{p^{\frac12}qp^{\frac12}}$, the spectral distribution of $p^{\frac12}qp^{\frac12}$ with respect to the (unique) trace $\tau$ on $L(\Gamma)$. Since $\tau((p^{\frac12}qp^{\frac12})^n)=\tau((pqp)^n)$ for every $n\in\N_0$, $\mu\boxtimes\nu=\mu_{pqp}$, the spectral distribution of $pqp$ w.r.t. $\tau$.
Define the self-adjoint operator 
$$S:=\left(
\begin{array}{cc}
 0 & (pq)^* \\
 pq & 0 \\
\end{array}
\right)\in M_2(L(\Gamma)),
$$
and let $\mu_{\mathrm{free}}:=\mu_S$ be the spectral distribution of $S$ with respect to the normal faithful trace $\tau\otimes\tau_2$, where $\tau_2:=\frac12 Tr$ is the normalized trace on $M_2(\C)$.
By \cite[Eq.~(2.6)]{HaagerupRamirezSolano}, $m_{2n}(S)=\tau(((pq)^*(pq))^{n})=\tau((qpq)^{n})=\tau((pqp)^{n})$.
Thus the image measure $\phi(\mu_S)$ of $\mu_S$ by the map $\phi:t\mapsto t^2$ is $\phi(\mu_{S})=\mu_{pqp}=\mu\boxtimes\nu$.
Substituting $t^2=x$ in Eq.~(\ref{e:muxnu}), we get
\begin{eqnarray*}
\mu_{S}(t)
&=&\frac1{\pi}\frac{\sqrt{(\lambda_1^2-t^2)(t^2-\lambda_2^2)}}{t(1-t^2)} 1_{[\lambda_2,\lambda_1]}(t) dt \,\,+ c_0\delta_0.
\end{eqnarray*}
Symmetrizing,
$$
\mu_{S}(t)=\frac 1{2\pi}\frac{\sqrt{(\lambda_1^2-t^2)(t^2-\lambda_2^2)}}{t(1-t^2)} 1_{[-\lambda_1,-\lambda_2]\cup[\lambda_2,\lambda_1]} dt+ c_0\delta_0,
$$and by
 \cite[Eq~(2.7)]{HaagerupRamirezSolano}, 
 $$||pq||=||S||=\max(\supp(\mu_S))=\lambda_1.$$  
\end{proof}
\begin{exam}[free case]\label{ex:freeCase}
Let $\Gamma:=\Z_2\star \Z_3$ be the free product with generators
$s\in \Z_2$, $c\in \Z_3$. Let $p=(e+s)/2$, $q=(e+c+c^2)/3\in \C\Gamma$. Then by Theorem \ref{t:freemeasure}
$$||pq||=\sqrt{\frac12(1-\frac13)}+\sqrt{\frac13(1-\frac12)}=\frac{1+\sqrt2}{\sqrt6}=0.985\ldots.$$
Let $\Gamma':=\Z_3\star \Z_4$ be the free product with generators $c\in \Z_3$, $d\in\Z_4$.
and let $q=(e+c+c^2)/3$, $r=(e+d+d^2+d^3)/4\in \C\Gamma'$. Then by Theorem \ref{t:freemeasure}
$$||qr||=\sqrt{\frac13(1-\frac14)}+\sqrt{\frac14(1-\frac13)}=\frac{\sqrt2+\sqrt3}{\sqrt{12}}=0.908\ldots.$$
\end{exam}

 \section{\textbf{Results}}\label{s:results}
 Let $P,Q,R\in L(T)$ be the projections given by
\begin{equation*}
P:=\frac{I+S}{2}\qquad Q:=\frac{I+C+C^2}{3}\qquad R:=\frac{I+D+D^2+D^3}{4},
\end{equation*}
where $S:=D^2$ and $C$, and $D$ are the generators of $T$ shown in Figure \ref{f:generatorsCDofT}.
Let $\tau$ be the trace on $\C T$ coming from the group von Neumann algebra $L( T)$, as described in \cite[Section 2]{HaagerupRamirezSolano}.
 Using the trace properties of $\tau$, the even moments for the self-adjoint operator
 \begin{equation}\label{e:tildeh}
\tilde h:=\sqrt{12}\left(
  \begin{array}{cc}
    0 & (QR)^* \\
    QR & 0 \\
  \end{array}
\right)\in M_2(C_r^*(T))
\end{equation}
are given by
$$m_{2n}(\tilde h) =\tau((\tilde h)^{2n}) =12^{n} \tau((QRQ)^{n})=12^n\tau((QR)^n),\quad n\in\N,$$
where 
$$h:=\sqrt{12}QR=\frac1{\sqrt{12}}(I+C+C^2)(I+D+D^2+D^3).$$
Our goal is to compute as many as possible of the moments
\begin{eqnarray}\label{e:mn}
\nonumber m_n&:=&m_n(h^*h)\\
\nonumber&=&m_{2n}(\tilde h)\\
&=&\tau(((I+C+C^2)(I+D+D^2+D^3))^n),\quad n\in\N,
\end{eqnarray}
because by Section 4 in \cite{HaagerupRamirezSolano}  we can use them to estimate the norm $|| h||$ (which is equal to $|| \tilde h||$):
$$\lim_{n\to\infty}m_n^{1/2n}=\lim_{n\to\infty}\sqrt{12}\tau((QRQ)^n)^{1/2n}= \sqrt{12}||QRQ||^{1/2}=\sqrt{12}||QR||=|| h||.$$

Rather than computing the moments $m_n$ directly, we compute the ``cyclic reduced" numbers 
$$\zeta_n:=\tau(((C+C^2)(D+D^2+D^3))^n),\quad n\in\N.$$
These numbers are related by the formula 
\begin{equation}\label{e:mn_zn}
\zeta_n - \frac{5+(-2)^n}{12}= 2 \sum_{k=0}^n \frac{c_{n,k}}{12^k} m_k, \qquad n\in \N_0,
\end{equation}
where $m_0:=1/4$, ($\zeta_0:=1$), and $c_{n,k}$ is the coefficient of $t^k$ in the substituted Chebyshev polynomial of the first kind $6^{n/2}T_n( \sqrt6(t-\frac5{12}))$. 
To obtain this formula (\ref{e:mn_zn}), we use Lemma~\ref{l:nu_mn} and Theorem~\ref{t:mn_zn} with $\alpha=\tau(Q)=1/3$,  $\beta=\tau(R)=1/4$, and note that from Lemma~\ref{l:nu_mn},  $\int_{[0,1]}t^k\,d\nu=2 m_k/12^k$.

To compute the cyclic reduced numbers $\zeta_n$, we use the inner product 
$$\langle x,y\rangle :=\tau(y^*x)$$ on $L(T)$ associated with the trace $\tau$. 
Recall that the trace restricted to the group ring $\C T$ (considered as a subalgebra of $C_r^*(T)\subset L(T)$)
satisfies
$$\tau(\sum_{x\in Y} c_x x)=  \left\{\begin{array}{cc}
    c_e & \mathrm{if }\,\, e\in Y  \\ 
    0 & \mathrm{else} \\ 
  \end{array}\right.,
$$ 
for any finite subset $Y\subset T$ and any set of complex numbers $(c_x)_{x\in T}$ indexed by $T$.
In particular,  $\tau(x)=1$ if $x=e$ else it is zero.
For instance, expanding $(I+C+C^2)(I+D+D^2+D^3)$ we see that the identity occurs only once and thus $m_1=1$.
Thus 
\begin{equation}\label{e:innerproduct}
\langle \sum_{x\in X} c_x\, x, \sum_{y\in Y} c_y\, y\rangle =  \sum_{x\in X}  \sum_{y\in Y} c_x\overline{c_y} \langle x, y\rangle,
\end{equation}
for any finite subsets $X,Y\subset T$, and any $c_x,c_y\in \C$.

Define
$$a:=C+C^2\qquad b:=D+D^2+D^3,$$
and recall that $C^3=I$, $D^4=I$, and that $a=a^*$ and $b=b^*$. Then
$$\zeta_n=\tau((ab)^n),\quad n\in\N.$$
Using computers, we compute $\zeta_n$  as follows.
We store in one file the terms of the expanded sum $(ab)^n$, (one term per line), and in another file those of $(ba)^n$. Here a term is a word of length $n$  in the letters $C,C^{-1}, D,D^{-1},D^2$.
Composition of the letters in the word is done using the algorithm described in section \ref{s:amenability}, which yields a reduced doubletree. This doubletree is then converted into a sequence of zeroes and ones and separators and pointers by serializing the range and domain trees  using the preorder traversal method. We save this sequence in base 64 in a single line.
To save time and space we read the file for $(ab)^n$ and apply $ab$ (i.e. we multiply each word with 
$CD,CD^2,CD^{-1},C^{-1}D,C^{-1}D^2$, and $C^{-1}D^{-1}$) in order to obtain the file for
$(ab)^{n+1}$. To save space we store not only the word but also its frequency.
Taking the inverse of each word  in the file for $(ab)^n$ gives the file for $(ba)^n$.
The inverse of a doubletree $x\in T$ with range tree $r(x)$ and domain tree $d(x)$ is the doubletree with range tree $d(x)$ and domain tree $r(x)$ (i.e. it simply swaps the domain and range trees).
The inner product
$$\zeta_{2n}=\langle (ab)^n, (ba)^n\rangle, \quad n\in\N$$
is the intersection of the sorted files for $(ab)^n$ and $(ba)^n$. We used the GNU sort program to sort the files.
The odd ``cyclic reduced" numbers $\zeta_{2n+1}$ are computed similarly
$$\zeta_{2n+1}=\langle (ab)^{n+1}, (ba)^n\rangle=\langle (ab)^{n}, (ba)^{n+1}\rangle,\quad n\in\N_0.$$
To reduce the size of the files by about one half, we compute instead the numbers
$\zeta_{n}^s, \zeta_{n}^e$ given in Eqs.~(\ref{e:zeta_se_even}),(\ref{e:zeta_se_odd});
the idea is the following: Since $(ab)^n=Cb(ab)^n+C^{-1}b(ab)^n$, and $b(ab)^n$ is self-adjoint, we can obtain the terms of the expanded sum for  $C^{-1}b(ab)^n$ from those of $Cb(ab)^n$ by using the ``reverse-inverse" map $\pi\circ J=Ad(D^2)\circ R$ given in  Eq.~(\ref{e:ri_EQ_D2_R_D2}). For instance the reverse-inverse of the word $CDC^2D^2$ is $D^2\,R(CDC^2D^2)\,D^2=C^{-1}D^{-1}C^{-2}D^{-2}$, which corresponds to reversing the order of the letters and taking the inverse. The letter $s$ (resp.~$e$) in $\zeta_{n}^s$ (resp.~$\zeta_{n}^e$) stands for that we are only keeping the words which are smaller than (resp.~equal to) their corresponding reverse-inverses.
The comparison is done using the lexicographic order of the serialized form of the doubletrees.
The relation between these numbers and the reduced cyclic number $\zeta_n$ is given by Lemma \ref{l:zeta_se_zetarel}
$$
\zeta_{n}=2\zeta_{n}^s+4 \zeta_{n}^e,\quad n\ge2.
$$
We wrote two programs in $C\#$ and Haskell, both using parallel programming, to calculate the numbers 
$\zeta_{n}^s, \zeta_{n}^e$, which can be downloaded at
\begin{verbatim}
https://github.com/shaagerup/ThompsonGroupT/ 
https://github.com/mariars/ThompsonGroupT/
\end{verbatim}
The size of each of the two files for computing $\zeta_{28}^s$ is about {\color{black}281} GB and {\color{black}285} GB. They were run in a desktop computer with 2 TB of SSD hard disk, and on the Abacus 2.0 supercomputer 
from the DeIC National HPC Centre.
The series of numbers $\zeta_n^s, \zeta_n^e,\zeta_n, m_n$ are shown in Table \ref{t:SeriesOfNumbers}.
\begin{table}[b]
\begin{tabular}{lllllll}
\hline
$n$ & $\zeta_n^s$ & $\zeta_n^e$ & $\zeta_n$ & $m_n(h^*h)$\\
\hline
1 & 0 & 0 & 0 & 1\\
2 & 0 & 0 & 0 & 6\\
3 & 0 & 0 & 0 & 42\\
4 & 0 & 0 & 0 & 318\\
5 & 1 & 0 & 2 & 2528\\
6 & 0 & 0 & 0 & 20790\\
7 & 0 & 0 & 0 & 175344\\
8 & 4 & 0 & 8 & 1508158\\
9 & 36 & 0 & 72 & 13177554\\
10 & 70 & 3 & 152 & 116636378\\
11 & 64 & 1 & 132 & 1043596346\\
12 & 524 & 1 & 1052 & 9423929906\\
13 & 2228 & 4 & 4472 & 85780131568\\
14 & 8160 & 16 & 16384 & 786252907282\\
15 & 21617 & 32 & 43362 & 7251162207110\\
16 & 81644 & 102 & 163696 & 67241091321510\\
17 & 279531 & 697 & 561850 & 626619942680948\\
18 & 1006816 & 4990 & 2033592 & 5865627675769158\\
19 & 3429416 & 13057 & 6911060 & 55130780282172364\\
20 & 12284412 & 35247 & 24709812 & 520110723876289138\\
21 & 43215686 & 89274 & 86788468 & 4923701716098043110\\
22 & 154863150 & 246102 & 310710708 & 46759540919860581346\\
23 & 550890233 & 763137 & 1104833014 & 445382340814268264936\\
24 & 1982133410 & 2484953 & 3974206632 & 4253954798148920432622\\
25 & 7128125209 & 9275681 & 14293353142 & 40735421620966779279998\\
26 & 25797672490 & 34858087 & 51734777328 & 391022235546378412228050\\
27 & 93561508424& 119608865&  187601452308  &  3761992784005490950198026\\
28 & 341014479116&411320336 &  683674239576  &  36271465945557216051920334\\
\hline
\end{tabular}
\caption{ The series of numbers for $h=\frac1{\sqrt{12}}(I+C+C^2)(I+D+D^2+D^3)$ .}
\label{t:SeriesOfNumbers}
\end{table}

\begin{table}[b]
\begin{tabular}{llllll}
\hline
$n$ &  $m_n^{\frac1{2n}}$ & $\sqrt{\frac{m_{n}}{m_{n-1}}}$& $\alpha_n$ & $\lambda_{\mathrm{max}}(M_n)$ & $\alpha_{n-1}+\alpha_n$\\
\hline
1 & 1.00000 & 2.00000 & 2.00000 & 2. & - - - - -\\
2 & 1.56508 & 2.44949 & 1.41421 & 2.44949 & 3.41421\\
3 & 1.86441 & 2.64575 & 1.73205 & 2.71519 & 3.14626\\
4 & 2.05496 & 2.75162 & 1.41421 & 2.82843 & 3.14626\\
5 & 2.18916 & 2.81952 & 1.77951 & 2.9224 & 3.19373\\
6 & 2.28992 & 2.86773 & 1.41111 & 2.97266 & 3.19062\\
7 & 2.36899 & 2.90414 & 1.75248 & 3.01653 & 3.16359\\
8 & 2.43306 & 2.93277 & 1.42622 & 3.04235 & 3.17870\\
9 & 2.48626 & 2.95593 & 1.78930 & 3.06842 & 3.21552\\
10 & 2.53129 & 2.97509 & 1.43494 & 3.08578 & 3.22424\\
11 & 2.57000 & 2.99123 & 1.75349 & 3.10309 & 3.18844\\
12 & 2.60371 & 3.00504 & 1.43767 & 3.11454 & 3.19117\\
13 & 2.63339 & 3.01701 & 1.78223 & 3.12665 & 3.21990\\
14 & 2.65975 & 3.02753 & 1.45087 & 3.13534 & 3.23310\\
15 & 2.68337 & 3.03685 & 1.76147 & 3.14458 & 3.21234\\
16 & 2.70467 & 3.04518 & 1.45124 & 3.15121 & 3.21271\\
17 & 2.72399 & 3.05270 & 1.77229 & 3.15841 & 3.22353\\
18 & 2.74163 & 3.05953 & 1.45749 & 3.16373 & 3.22978\\
19 & 2.75780 & 3.06577 & 1.76751 & 3.16956 & 3.22500\\
20 & 2.77270 & 3.07150 & 1.46238 & 3.17394 & 3.22989\\
21 & 2.78647 & 3.07679 & 1.76938 & 3.1788 & 3.23176\\
22 & 2.79926 & 3.08169 & 1.46384 & 3.18249 & 3.23321\\
23 & 2.81116 & 3.08625 & 1.76757 & 3.1866 & 3.23141\\
24 & 2.82228 & 3.09051 & 1.46904 & 3.18976 & 3.23662\\
25 & 2.83269 & 3.09449 & 1.76906 & 3.19332 & 3.23810\\
26 & 2.84247 & 3.09824 & 1.47004 & 3.19607 & 3.23909\\
27 & 2.85168 & 3.10176 & 1.76808 & 3.19917 & 3.23811\\
28 & 2.86036 & 3.10509 & 1.47534 & 3.2016 & 3.24341\\\hline
\end{tabular}
\caption{ Estimating the norm $\frac1{\sqrt{12}}||(I+C+C^2)(I+D+D^2+D^3)||$ .}
\label{t:rootsRationsAlphasNorms}
\end{table}

\begin{figure}
  \includegraphics[scale=.6]{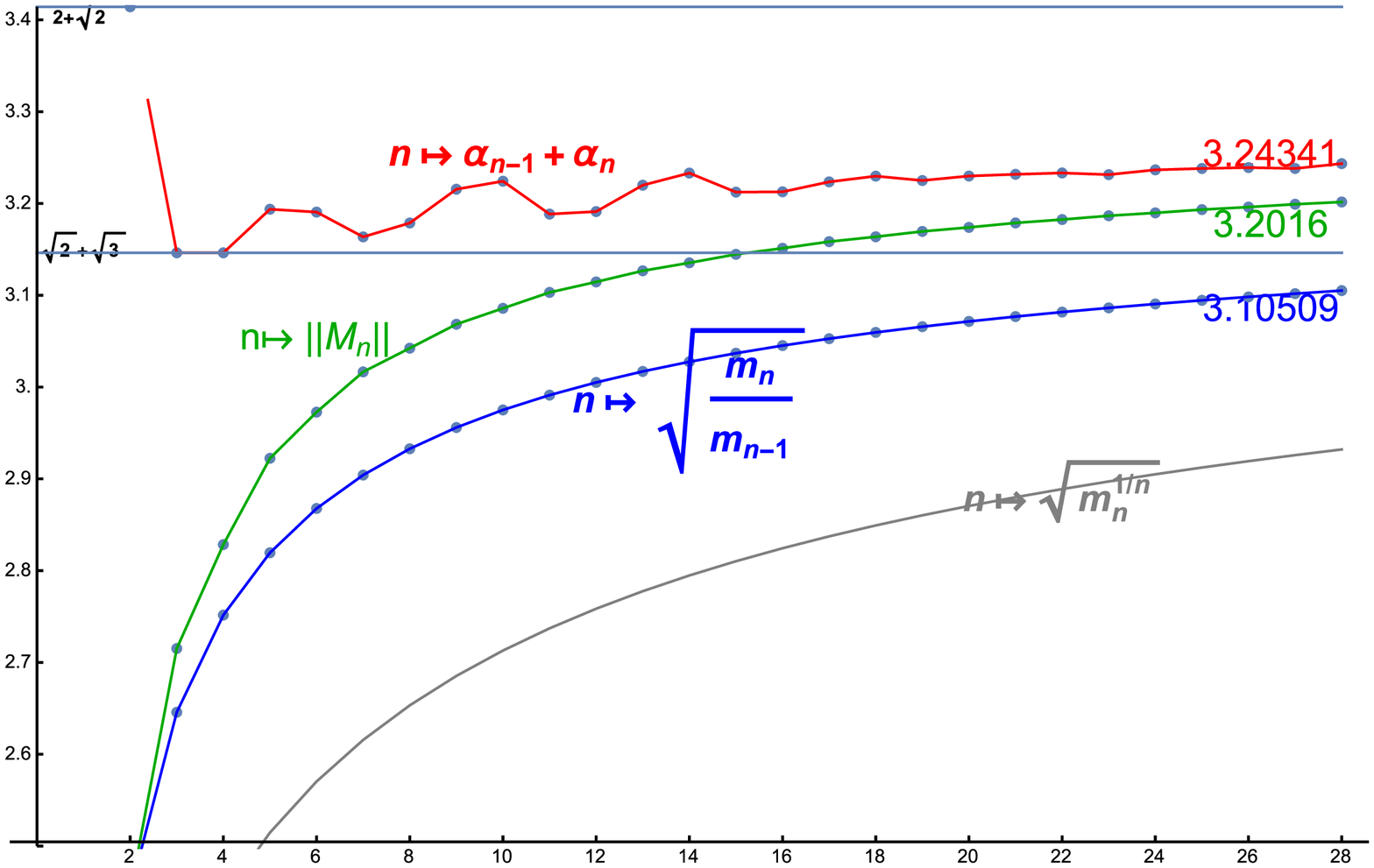}\\
  \caption{Estimating the norm $\frac{1}{\sqrt{12}}||(I+C+C^2)(I+D+D^2+D^3)||.$}\label{f:Mnnorms}
\end{figure}

In comparison, when one considers the free product $\Z_3\star \Z_4$ on two generators $c\in\Z_3$, $d\in \Z_4$,
the measure $\mu_{\text{free}}$ in Eq.~(\ref{e:mufree})  based on $c$, $d$ instead of $C$, $D$, is computed explicitly using Theorem \ref{t:freemeasure}
with $\alpha=\tau(\frac{e+c+c^2}3)=1/3$, $\beta=\tau(\frac{e+d+d^2+d^3}4)=1/4$, (and substituting $x=\sqrt{12}\,t$). Hence
the moments $m_n^{\text{free}}$ can be computed explicitly with Eq.~(\ref{e:mufree}) together with Lemma \ref{l:mn_free}.
One gets
$m^{\text{free}}_1=1$, and
$$ m^{\text{free}}_n=12 m^{\text{free}}_{n-1}-6\sum_{j=0}^{\left\lfloor \frac{n-2}{2}\right\rfloor} 
6^j  5^{-2 j+n-2}\frac{\binom{2 j}{j}}{j+1} \binom{n-2}{2 j}, \qquad n\ge2.
$$

Estimations of the norm $||h||$ are done as follows.
By \cite[Section 4]{HaagerupRamirezSolano} and the theory of orthogonal polynomials we can define a sequence of positive numbers 
$$\alpha_n:=\left(\frac{D_{n-2}D_n}{D_{n-1}^2}\right)^{1/2}, \qquad n\in\N,$$
where 
$$D_{-1}:=1,\quad D_n:=\det([m_{i+j}(\tilde h)]_{i,j=0}^n).$$
For instance $$D_4=\det\left(
\begin{array}{ccccc}
 m_0 & 0 & m_1 & 0 & m_2 \\
 0 & m_1 & 0 & m_2 & 0 \\
 m_1 & 0 & m_2 & 0 & m_3 \\
 0 & m_2 & 0 & m_3 & 0 \\
 m_2 & 0 & m_3 & 0 & m_4 \\
\end{array}
\right),$$
where $m_0:=1/4$ and $m_n$ is defined in Eq.~(\ref{e:mn}).
Let $$M_n=\left(
     \begin{array}{ccccc}
              0 & \alpha_1 &         &  & \\
       \alpha_1 & 0        & \alpha_2&  &0 \\
                & \alpha_2 & \ddots  & \ddots &\\
                &          &  \ddots & \ddots  &\alpha_n\\
       0        &          &         &\alpha_n & 0 \\
     \end{array}
   \right).$$

Then by \cite[Proposition 4.4]{HaagerupRamirezSolano}
  \begin{equation}\label{e:alphas_h}
    \liminf_{n\to\infty}(\alpha_{n-1}+\alpha_n)\quad \le \quad ||h||\quad \le\quad  \sup_{n\ge 3}(\alpha_{n-1}+\alpha_n).
  \end{equation}

Moreover by \cite[Proposition 4.1 and Proposition 4.7]{HaagerupRamirezSolano},
the roots $(m_n^{\frac1{2n}})_{n=1}^{\infty}$, the ratios $(\sqrt{\frac{m_n}{m_{n-1}}}\,)_{n=1}^{\infty}$,
and the norms $(||M_n||)_{n=1}^{\infty}$ are increasing sequences that converge to $||h||$ and satisfy
$$m_n^{\frac1{2n}}\le \sqrt{\frac{m_n}{m_{n-1}}}\le ||M_n||\le \frac1{\sqrt{12}}||(I+C+C^2)(I+D+D^2+D^3)||,\quad n\in \N.$$

We list these sequences in Table \ref{t:rootsRationsAlphasNorms}, and plot them in Figure \ref{f:Mnnorms} for $n=1,\ldots, 28$. 
When $n\ge 16$, the norm $||h||\ge \sqrt{2}+\sqrt{3}$. The best lower bound for $h$ that we can obtain from our data is
$$\frac1{\sqrt{12}}||(I+C+C^2)(I+D+D^2+D^3) ||\ge 3.2016,$$
and a very likely lower bound for $h$ is
$$\frac1{\sqrt{12}}||(I+C+C^2)(I+D+D^2+D^3) ||\ge \alpha_{27}+\alpha_{28}=3.24341,$$
because of Eq~(\ref{e:alphas_h}) and because the sequence $(\alpha_{n-1}+\alpha_n)_{n=1}^\infty$ appear to be monotonically increasing for $n\ge 19$.
By making a least squares fitting of the numbers $\lambda^2_{\max}(M_{18}),\ldots,\lambda^2_{\max}(M_{28})$ to a function of the form $f(n)=a-b(n-c)^{-d}$  we get
$$f(n)={10.79\, -\frac{8.952}{(n+0.2)^{0.841}}}.$$
In particular, the extrapolation method suggests that 
\begin{equation}\label{e:extrapolation3.28}
\frac1{\sqrt{12}}||(I+C+C^2)(I+D+D^2+D^3) ||\approx \sqrt{10.79}=3.28.
\end{equation}
This and Theorem \ref{t:QRleQP} suggest that $F$ might be non-amenable.

 

 \subsection{Spectral distribution measure}
Recall that $h=\sqrt{12}QR$.
Since $\tilde h$ defined in Eq.~(\ref{e:tildeh}) is self-adjoint, we get by the spectral theorem that
there is a unique probability measure $\mu_{\tilde h}$ to $\R$ with support (cf.~\cite[Section 2]{HaagerupRamirezSolano})
$$\supp(\mu_{\tilde h})\subset [-||\tilde h||,||\tilde h||]\subset [-(1+q),(1+q)],$$
where $||\tilde h||=||h||\le 2+\sqrt{2}=1+q$ by Theorem \ref{t:QRleQP}. 
 In this subsection, we will estimate the spectral distribution measure $\mu_{\tilde h}$, as it was done in \cite[Section 5]{HaagerupRamirezSolano}.
The Hilbert space $L^2([-(q+1),(q+1)],dt)$ can be equipped with the orthonormal basis 
$$\sqrt{\frac{n+\frac12}{1+q}}P_n(\frac{t}{q+1}),\quad n\in \N_0,$$
where $P_n$, $n\in \N_0$, are the Legendre polynomials.
The Hilbert space $L^2([-(q+1),q+1],\frac{1}{\pi\sqrt{(1+q)^2-t^2}}dt)$ can be equipped with the orthonormal basis
$$T_0(\frac{t}{q+1}),\quad \sqrt{2}T_n(\frac{t}{q+1}),\quad n\in\N,$$
where $T_n$, $n\in\N_0$ are the Chebyshev polynomials of the first kind.

By  \cite[Eq. (5.3)]{HaagerupRamirezSolano}, the density of $\mu_{\tilde h}$ with respect to the Lebesgue measure can be approximated by
$$\rho'_N(t)=\sum_{n=0}^{2N} \frac{n+\frac12}{q+1} \left( \int_{-(q+1)}^{q+1} P_n(\frac{s}{q+1})d \mu_{\tilde h}(s)\right)P_n(\frac{t}{1+q})$$
and by
\begin{equation}\label{e:density}
 \rho_N(t)=\sum_{n=0}^{2N} c_n^2\left( \int_{-(q+1)}^{q+1} T_n(\frac{s}{q+1})d \mu_{\tilde h}(s)\right) T_n(\frac{t}{1+q})\frac{1}{\pi\sqrt{(1+q)^2-t^2}},
 \end{equation}
 where $c_0:=1$ and $c_n:=\sqrt{2}$, $n\in\N$.
 Since $\mu_{\tilde h}$ is a symmetric measure, all the odd terms in Eq.~(\ref{e:density}) are zero.
Using the moments $m_n$  in Eq.~(\ref{e:mn}) (with $m_0=1/4$) and the formula 
$$\int_{-(q+1)}^{q+1} t^{2n}\,d\mu_{\tilde h}=m_{2n}(\tilde h)=m_n,\qquad n\in\N_0,$$
we compute the ``Chebyshev" density $\rho_{28}$ for $0\le t\le 2+\sqrt{2}$ and plot it in Figure \ref{f:rho28}, together with the corresponding ``free" density. We are more interested in the tail of the measure, however, because it gives us an estimate of the norm $||\tilde h||=||h||$:
 $$||h||=\max\{|t|\,\mid t\in\supp(\mu_{\tilde h})\},$$
 (cf.~\cite[Section 2]{HaagerupRamirezSolano}).
 We plot $\rho_{28}$ in the tail interval $[\sqrt2+\sqrt3,\,\,2+\sqrt2]$ in Figure \ref{f:rho28tails}.
This shows that $\mu_{\tilde h}$ has very little mass in $[3.22,3.414]$, hence $||h||$ can be any number in $[3.22,3.414]$ including the extrapolated number $3.28$ found in Eq.~(\ref{e:extrapolation3.28}).

\begin{figure}
  \includegraphics[scale=.4]{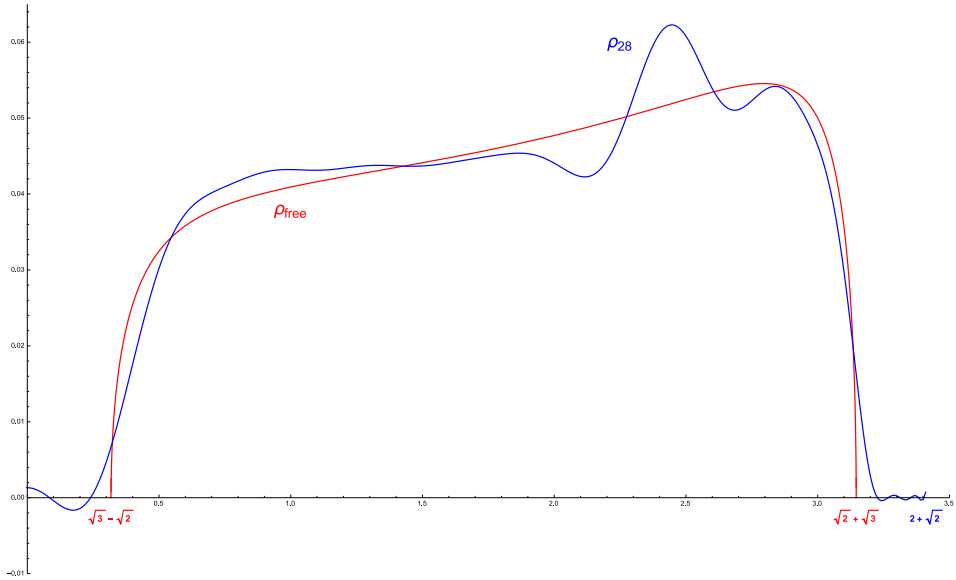}\\
  \caption{Estimating the Chebyshev density for $\mu_{\tilde h}$ where $h=\frac{1}{\sqrt{12}}(I+C+C^2)(I+D+D^2+D^3)$.}\label{f:rho28}
\end{figure}

\begin{figure}
  \includegraphics[scale=.5]{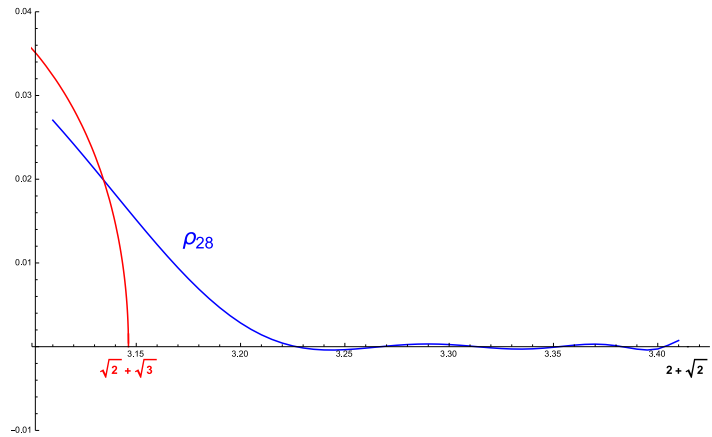}\\
  \caption{Estimating the tail of Chebyshev density for $\mu_{\tilde h}$ where  $h=\frac{1}{\sqrt{12}}(I+C+C^2)(I+D+D^2+D^3)$.}\label{f:rho28tails}
\end{figure}



\subsection*{Acknowledgments.}
We are very grateful with the DeIC National HPC Centre for allowing us to use the supercomputer Abacus 2.0 to run some of our computer code.
The first and third author would like to thank  Steen Thorbj\o rnsen, Erik Christensen, and Wojciech Szymanski for the nice discussions that they had with them.
The second author was supported by the Villum Foundation under the project ``Local and global structures of groups and their algebras" at University of Southern Denmark, and by the ERC Advanced Grant no. OAFPG 247321, and partially supported by the Danish National Research Foundation (DNRF) through the Centre for Symmetry and Deformation at University of Copenhagen, and the Danish Council for Independent Research, Natural Sciences.
 The third author was supported by the Villum Foundation under the project ``Local and global structures of groups and their algebras" at University of Southern Denmark, and by the ERC Advanced Grant no. OAFPG 247321, and by the Center for Experimental Mathematics at University of Copenhagen.

\appendix

\section{\textbf{free moments}}
\begin{lem}\label{l:mn_free}
Let
\begin{equation}\label{e:intfree} 
m_n^{\mathrm{free}}:=\int_{\sqrt{3}-\sqrt{2}}^{\sqrt{3}+\sqrt{2}} \frac{t^{2n}}{\pi}\frac{\sqrt{24-(t^2-5)^2}}{t(12-t^2)}dt,\quad n\in\N.
\end{equation}
Then $m_1^{\mathrm{free}}=1$ and 
$$m_{n+1}^{\mathrm{free}}=12m_n^{\mathrm{free}}-6\sum_{j=0}^{\lfloor {\frac{n-1}2} \rfloor} C_j\binom{n-1}{2j}6^j 5^{n-1-2j},\quad n\in\N$$
where $C_j:=\frac{1}{j+1}\binom{2j}j$, $j\in\N_0$ are the Catalan numbers.
\end{lem}

\begin{proof}
Substituting   $x=t^2$ in Eq.(\ref{e:intfree}), we get
$$m_n^{\mathrm{free}}=\int_{5-2\sqrt{6}}^{5+2\sqrt{6}} \frac{x^{n-1}}{2\pi}\frac{\sqrt{24-(x-5)^2}}{12-x}dx,\quad n\in\N.$$
Then
$$12m_{n+1}^{\mathrm{free}}-m_{n+2}^{\mathrm{free}}=\frac{1}{2\pi}\int_{5-2\sqrt{6}}^{5+2\sqrt{6}} x^{n}\sqrt{24-(x-5)^2}dx,\quad n\in\N.$$
Substituting $\frac{x-5}{\sqrt{6}}=y$, we get
\begin{equation}\label{e:mdiff}
12m_{n+1}^{\mathrm{free}}-m_{n+2}^{\mathrm{free}}=\frac{1}{2\pi}\int_{-2}^{2} 6(\sqrt{6}y+5)^{n}\sqrt{4-y^2}dy,\quad n\in\N.
\end{equation}
Using the binomial formula, we get
$$12m_{n+1}^{\mathrm{free}}-m_{n+2}^{\mathrm{free}}=6\sum_{k=0}^{n}\frac{1}{2\pi}\int_{-2}^{2} y^k\sqrt{4-y^2}dy\, \binom{n}{k}6^{k/2}5^{n-k} ,\quad n\in\N.$$
Observing that the integral is zero when $k$ is odd,
$$12m_{n+1}^{\mathrm{free}}-m_{n+2}^{\mathrm{free}}=6\sum_{j=0}^{\lfloor\frac n 2 \rfloor}\left(\frac{1}{2\pi}\int_{-2}^{2} y^{2j}\sqrt{4-y^2}dy\right)\, \binom{n}{2j}6^{j}5^{n-2j} ,\quad n\in\N.$$
Hence
$$12m_{n+1}^{\mathrm{free}}-m_{n+2}^{\mathrm{free}}=6\sum_{j=0}^{\lfloor\frac n 2 \rfloor}C_j\, \binom{n}{2j}6^{j}5^{n-2j} ,\quad n\in\N,$$
where $C_j:=\frac{1}{j+1}\binom{2j}j$ is a Catalan number.\\
Eq.~(\ref{e:mdiff}), holds also  if we let $n=-1$:
$$12m_{0}^{\mathrm{free}}-m_{1}^{\mathrm{free}}=\frac{6}{2\pi}\int_{-2}^{2} \frac{\sqrt{4-y^2}}{\sqrt{6}y+5}dy=2.$$
It follows that $m^{\mathrm{free}}_1=1$ because by Eq.~(\ref{e:intm0}) (with $\alpha=1/3$, $\beta=1/4$),  $m^{\mathrm{free}}_0=1/4$.

\end{proof}

\section{\textbf{Estimating the norm $||C+D+C^{-1}+D^{-1}||$}}\label{a:normCDCinvDinv}
\begin{sidewaystable} 
\centering
$ $\\$ $\\$ $\\$ $\\$ $\\$ $\\
$ $\\$ $\\$ $\\$ $\\$ $\\$ $\\
$ $\\$ $\\$ $\\$ $\\$ $\\$ $\\
$ $\\$ $\\$ $\\$ $\\$ $\\$ $\\
$ $\\$ $\\$ $\\$ $\\$ $\\$ $\\
$ $\\$ $\\$ $\\$ $\\$ $\\$ $\\
\begin{tabular}{lllll}
\hline
$n$ & $||h_n||^2$ & $\eta_n$ & $\zeta_n$ & $m_n(h^*h)$ \\
\hline
1 & 4 & 0 & 0 & 4 \\
2 & 14 & 2 & 2 & 30 \\
3 & 46 & 6 & 2 & 270 \\
4 & 182 & 50 & 34 & 2678 \\
5 & 856 & 360 & 244 & 28418 \\
6 & 3888 & 1680 & 844 & 317246 \\
7 & 20536 & 10552 & 6356 & 3681822 \\
8 & 111366 & 60310 & 35010 & 44027350 \\
9 & 642550 & 368762 & 222842 & 538815546 \\
10 & 3850086 & 2291198 & 1407754 & 6714321830 \\
11 & 23444600 & 14185540 & 8719700 & 84869473770 \\
12 & 145705728 & 89557468 & 55720548 & 1085055369622 \\
13 & 915645208 & 568085492 & 355133636 & 14001672259722 \\
14 & 5816241006 & 3637390874 & 2288268034 & 182071429751606 \\
15 & 37273096250 & 23461764106 & 14837859518 & 2382930531465042 \\
16 & 240608480566 & 152250955922 & 96703523122 & 31360608130235654 \\
17 & 1563526262404 & 993951776628 & 633902431984 & 414711515674495370 \\
18 & 10219209908952 & 6522582898368 & 4174630000468 & 5507403086681142854 \\
19 & 67146704535028 & 43011657706540 & 27618539011904 & 73415226964469375622 \\
20 & 443323828665766 & 284895372767222 & 183478938659506 & 981973882890399349286 \\
21 & 2939893937656674 & 1894817824426598 & 1223610644784438 & 13175045740884220099018 \\
22 & 19575351631144042 & 12650487642600618 & 8189644814105262 & 177267112861509055927594 \\
23 & 130835022206113204 & 84759454955281696 & 54997636841585104 & 2391279755795301975623294 \\
24 & 877529231836455728 & 569783620173397812 & 370502892149137828 & 32335124616320091148224950 \\
25 & 5905019922806515884 & 3842215847470546512 & 2503367879099490904 & 438214977894105234044150738\\
26 & 39857811116595156626 & 25984967195646155486 & 16961687532334006854 & 5951190674684154918623110822\\
27 & 269809546538449104054 & 176221080384309789662 & 115227866329705330058 & 80978038411680591548914827558\\
28 & 1831388211478414017418 & 1198180652247376494918 & 784745277424152455990 & 1103891232023903090341589166522\\
\hline
\end{tabular}
\caption{ The series of numbers for $h=C+C^{-1}+D+D^{-1}$ .}
\label{t:Case2SeriesOfNumbers}
\end{sidewaystable}

\begin{table}[b]
\begin{tabular}{llllll}
\hline
$n$ &  $m_n^{\frac1{2n}}$ & $\sqrt{\frac{m_{n}}{m_{n-1}}}$& $\alpha_n$ & $\lambda_{\mathrm{max}}(M_n)$ & $\alpha_{n-1}+\alpha_n$\\
\hline
1 & 2.00000 & 2.00000 & 2.00000 & 2. & - - - - -\\
2 & 2.34035 & 2.73861 & 1.87083 & 2.73861 & 3.87083\\
3 & 2.54230 & 3.00000 & 1.79284 & 3.0557 & 3.66367\\
4 & 2.68211 & 3.14937 & 1.94854 & 3.25861 & 3.74139\\
5 & 2.78843 & 3.25755 & 2.04888 & 3.42926 & 3.99743\\
6 & 2.87375 & 3.34119 & 1.72771 & 3.51875 & 3.77659\\
7 & 2.94445 & 3.40670 & 1.96392 & 3.57859 & 3.69163\\
8 & 3.00423 & 3.45804 & 1.78580 & 3.61369 & 3.74972\\
9 & 3.05548 & 3.49831 & 1.94497 & 3.63956 & 3.73078\\
10 & 3.09992 & 3.53005 & 1.99819 & 3.66407 & 3.94316\\
11 & 3.13878 & 3.55529 & 1.75237 & 3.68118 & 3.75056\\
12 & 3.17305 & 3.57561 & 2.02259 & 3.69704 & 3.77496\\
13 & 3.20348 & 3.59223 & 1.79177 & 3.70881 & 3.81436\\
14 & 3.23068 & 3.60604 & 1.95285 & 3.71884 & 3.74462\\
15 & 3.25515 & 3.61772 & 1.98142 & 3.72895 & 3.93427\\
16 & 3.27727 & 3.62774 & 1.75239 & 3.73636 & 3.73381\\
17 & 3.29738 & 3.63648 & 2.03111 & 3.74354 & 3.78351\\
18 & 3.31575 & 3.64418 & 1.81639 & 3.74935 & 3.84750\\
19 & 3.33261 & 3.65107 & 1.97352 & 3.75482 & 3.78992\\
20 & 3.34813 & 3.65727 & 1.93786 & 3.7603 & 3.91138\\
21 & 3.36249 & 3.66291 & 1.78748 & 3.76455 & 3.72534\\
22 & 3.37581 & 3.66807 & 2.02147 & 3.76876 & 3.80895\\
23 & 3.38821 & 3.67283 & 1.82478 & 3.77227 & 3.84625\\
24 & 3.39979 & 3.67724 & 2.00115 & 3.77583 & 3.82593\\
25 & 3.41062 & 3.68134 & 1.88660 & 3.77924 & 3.88775\\
26 & 3.42079 & 3.68518 & 1.83914 & 3.78205 & 3.72575\\
27 & 3.43036 & 3.68877 & 2.00637 & 3.7849 & 3.84551\\
28 & 3.43939 & 3.69215 & 1.82673 & 3.7873 & 3.83309\\
\hline
\end{tabular}
\caption{ Estimating the norm $||C+C^{-1}+D+D^{-1}||$ .}
\label{t:rootsnorms}
\end{table}

\begin{figure}
  \includegraphics[scale=.5]{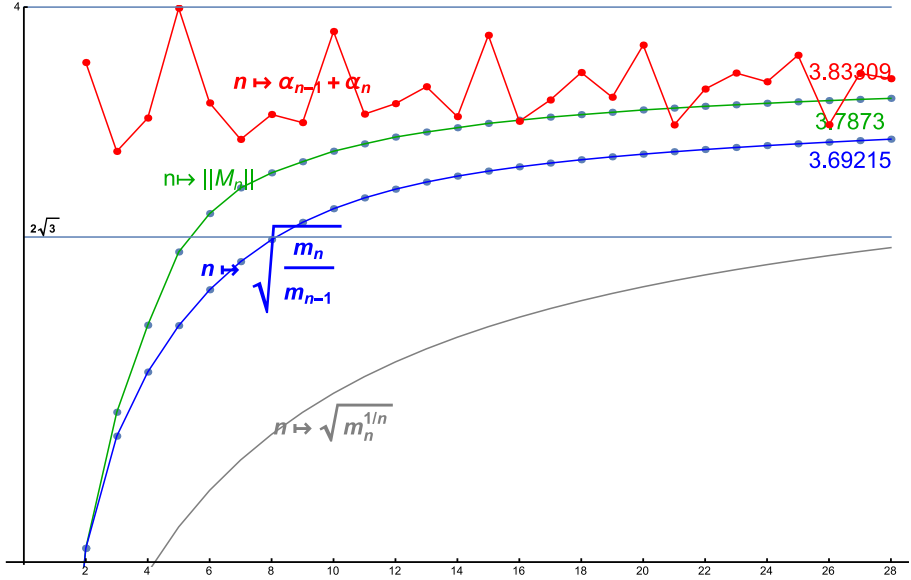}\\
  \caption{Estimating the norm $||C+D+C^{-1}+D^{-1}||.$}\label{f:MnnormsCDCinvDinv}
\end{figure}

Here we will describe briefly our estimation of the norm of
$$h:=C+D+C^{-1}+D^{-1}\in \C T,$$
where $C,D$ are the generators of $T$ whose graphs are shown in Figure \ref{f:generatorsCDofT}.
We use the same procedure as we did in \cite{HaagerupRamirezSolano}.
By \cite[Theorem 1.3]{HaagerupRamirezSolano}
 $$2\sqrt3< ||C+D+C^{-1}+D^{-1}||<4.$$
 The upper bound is never attained because the Thompson group $T$ is not amenable.
 We compute the first 28 even moments $m_n(h^*h)$ by first computing the sequence (with $q=3$)
 \begin{eqnarray*}
 &&h_0=e\\
 &&h_1=h\\
 &&h_2=h h_1-(q+1) h_0\\
 &&h_{n+1}=h h_n - q h_{n-1},\quad n\ge 2.
 \end{eqnarray*}
 Then we compute the sequences
 \begin{eqnarray*}
\xi_n&=&||h_n||_2^2-(q+1)q^{n-1},\\
   \eta_n&=&\xi_n-(q-1)(\xi_{n-1}+\xi_{n-2}+\cdots+\xi_1),\\
  \zeta_n&=&\eta_n-(q-1)(\eta_{n-1}+\eta_{n-2}+\cdots+\eta_1)\\
   m_n&=&\binom{2n}{n} q^n+\sum_{k=1}^{n}\binom{2n}{n-k} (\zeta_k+1-q)q^{n-k},\qquad n\in\N.
 \end{eqnarray*}
  
These series of numbers are listed in Table \ref{t:Case2SeriesOfNumbers}.
Next, we compute the increasing sequences of ``roots", ``ratios", ``norms" that converge to the norm (cf.~Section \ref{s:results}). We list them in Table \ref{t:rootsnorms} and plot them in Figure \ref{f:MnnormsCDCinvDinv}.
The best lower bound for $||h||$ that we can obtain from our results   is
$$||C+D+C^{-1}+D^{-1}||\ge 3.7873.$$
 By making a least squares fitting of the numbers $\lambda^2_{\max}(M_{18}),...,\lambda^2_{\max}(M_{28})$ to a function of the form $f(n)=a-b(n-c)^{-d}$ we get
$$f(n) = 14.8\, -\frac{18.5975}{(n+0.2)^{1.1097}}.$$
In particular, the extrapolation method predicts that
$$||C+D+C^{-1} +D^{-1} ||\approx \sqrt{14.8}=3.84,$$
 which is  closer to $4$.

$ $

\Addresses

\end{document}